\documentclass[10pt,leqno]{article}
\usepackage{amsmath,amsthm}
\usepackage{amsfonts}
\usepackage{mathrsfs}
\usepackage{hyperref}
\usepackage{amssymb}
\usepackage{epsfig}
\usepackage{subfig}
\usepackage{graphicx}
\usepackage{bm}
\usepackage{booktabs}
\usepackage{float}
\numberwithin{equation}{section} 
\newtheorem{theorem}{\bf Theorem}[section]
\newtheorem{example}{\bf Example}[section]

\newtheorem{remark}{\bf Remark}[section]
\newtheorem{lemma}{\bf Lemma}[section]

\newcommand{\norm}[1]{\left\lVert #1\right\rVert}

\newsavebox{\savepar}

\begin{document}
%\linenumbers
\title{\bf \LARGE Finite element theta schemes  for the viscous Burgers' equation with nonlinear Neumann boundary feedback control}
\author{
Shishu Pal Singh\thanks{Department of Mathematical Sciences, Rajiv Gandhi Institute of Petroleum Technology. Email: shishups@rgipt.ac.in}
\;and
	Sudeep Kundu\thanks{Department of Mathematical Sciences, Rajiv Gandhi Institute of Petroleum Technology. Email: sudeep.kundu@rgipt.ac.in}  
}
\date{\today}
\maketitle	
\begin{abstract} In this article, we develop a fully discrete numerical scheme for the one-dimensional (1D) and two-dimensional (2D) viscous Burgers’ equations with nonlinear Neumann boundary feedback control. The temporal discretization employs a $\theta$-scheme, while a conforming finite element method is used for the spatial approximation. The existence and uniqueness of the fully discrete solution are established. We further prove that the scheme is unconditionally exponentially stable for $\theta \in \left[\frac{1}{2}, 1\right]$, thereby ensuring that the stabilization property of the continuous model is retained at the discrete level. In addition, optimal error estimates are obtained for both the state variable and the boundary control inputs in 1D and 2D frameworks. Finally, several numerical experiments are presented to validate our theoretical findings and to demonstrate the effectiveness of the proposed stabilization strategy under varying model parameters.
\end{abstract}
\section{Introduction.}
We consider  the \(1D\) viscous Burgers' equation with Neumann boundary feedback control  in the following form:
\begin{align}
\label{Eq1}
y_{t}-\nu y_{xx} + yy_{x} &= 0, \quad (x,t)\in(0,1)\times(0,\infty) , \\
y_{x}(0,t) &= u_{0}(t),  \quad t\in(0,\infty),  \\
y_{x}(1,t) &= u_{1}(t), \quad t\in(0,\infty),	 \\
y\left(x,0\right) &= y_{0}(x), \quad x\in(0,1),
\label{Eq2}
\end{align}
where $\nu>0$ is the viscosity coefficient; $u_{0}(t)$ and  $u_{1}(t)$ are control inputs.

The viscous Burgers' equation is a semi-linear parabolic partial differential equation(PDE). In the absence of the viscous term \((\nu=0)\), it becomes a hyperbolic type PDE. Therefore, this form is a prototype for conservation equations that can develop shock waves. In this case, it is also known as the inviscid Burgers' equation.

Over the last few decades, the viscous Burgers' equation with control problems has gained much attention due to its application in various fields, like turbulence modeling, shallow water theory, gas dynamics, and traffic flow. The local stabilization results for this equation, under certain initial conditions, are found in \cite{MR1662924, Burns1991, MR1269994, MR1652947}. In the work \cite{MR1442024}, Ly et al. lifted these restrictions. They discussed existence, uniqueness, and regularity results for the viscous Burgers' equation with Dirichlet and Neumann boundaries for any initial data in \(L^{2}\).
For the \(2D\) viscous Burgers' equation, local stabilization is established in \cite{wasim2022, MR3348411, MR2744156}. In \cite{wasim2022}, Wasim et al. have shown local stabilization of the viscous Burgers' equation with a memory term, using distributed feedback control around a zero stationary solution. Additionally, the Dirichlet boundary control and the Neumann boundary control are used in the works \cite{MR3348411} and \cite{MR2744156}, respectively.

Applying a control Lyapunov functional, in \cite{MR1751258}, Krstic has demonstrated global stabilization of this equation with both Neumann and Dirichlet boundary feedback control laws. Next, \(H^{1}\) stabilization and regularity results are established in \cite{balogh1999global}. Stabilization for the generalized viscous Burgers' equation, in both adaptive (when \(\nu\) is unknown) and non-adaptive cases (when \(\nu\) is known), is shown in \cite{liu2001adaptive, MR2091456, MR2146486}. Moreover, stabilization using backstepping and adaptive backstepping boundary control appears, respectively, in \cite{liu2000backstepping} and \cite{jurado2023adaptive}. Finally, global stabilization of the \(2D\) viscous Burgers' equation with Neumann boundary feedback control is demonstrated in \cite{MR4139150}, using a control Lyapunov functional. Furthermore, the results related to optimal control problems are analyzed in \cite{MR1706822, MR3465457, MR3521259, MR1818917, zhu2006optimal}.

The numerical solution of the uncontrolled Burgers' equation is widely recognized. This equation with a non-linear Neumann boundary has been solved in \cite{MR1160632} using a three-time level splitting-up technique and the cubic spline method. For the Neumann boundary and mixed boundary conditions, this equation is solved numerically by different kinds of methods like a finite element method \cite{nguyen2001numerical, pugh1995finite}, a high order splitting method \cite{SEYDAOGLU2016410}, a probabilistic approach \cite{milstein2002probabilistic}, a proper orthogonal decomposition \cite{abbasi2015nonlinear}, and a second-generation wavelet optimized finite difference method \cite{MR3820676}, etc. In the work \cite{MR930029}, the author shows the behavior of the solution with a non-linear Neumann boundary condition. Additionally, the stability and instability of the stationary solution are discussed. 
 
The results related to the numerical simulation of the controlled viscous Burgers' equation are well established.
In \cite{balogh1999global}, the authors use a third-order quadratic approximation in time and a second-order finite difference method in space. They show that the uncontrolled system (zero Neumann boundary) converges to a constant stationary solution with cubic nonlinear boundary feedback control. In \cite{MR1269994}, a Galerkin approximation with linear spline basis functions shows the asymptotic behavior of the viscous Burgers' equation with Dirichlet boundary conditions. Smaoui, in the work \cite{MR2091456} and \cite{MR2146486}, uses the Chebyshev collocation method for spatial discretization and a backward Euler method in time. He shows that, in the non-adaptive case, the controlled system for the generalized viscous Burgers' equation is exponentially stable while in the adaptive case, it is asymptotically stable. Burns, in \cite{Burns1991}, employs a finite element method to stabilize the closed-loop nonlinear system by finding the linear feedback gain functional.

The global stabilization of this equation with Neumann boundary feedback control is discussed in \cite{MR3790146}. Further, the authors analyze a \(C^{0}\)-conforming finite element method in the spatial direction, keeping the time variable continuous, and global stabilization for the semi-discrete scheme is established. Moreover, optimal error is estimated for the state variable in the \(L^{\infty}(L^{\infty})\), \(L^{\infty}(L^{2})\), and \(L^{\infty}(H^{1})\)-norms, and error analysis for the feedback control laws is also established. In \(2D\) case, this equation is discretized in spatial direction using a \(C^{0}\)-conforming finite element method \cite{MR4139150}. Additionally, error is discussed for the state variable in \(L^{\infty}(L^{2})\) and \(L^{\infty}(H^{1})\)-norms, and a convergence result is established for the feedback control law.
  In this article, we present a theta scheme for the time discretization and keep a $C^{0}$-conforming finite element method in the spatial direction for this equation. Moreover, error estimates are derived for the state variable and control inputs of a fully discrete scheme.

The major contribution of this article as follows:
\begin{itemize}
	\item A theta scheme is applied for the fully discrete scheme, and the existence and uniqueness are established for this scheme.
	\item We discuss the exponential stability of a fully discrete scheme. The scheme is unconditionally stable when \(\theta\in[\frac{1}{2}, 1]\).
	\item We establish error analysis for the state variable and control inputs of a fully discrete scheme  when \( \theta \in \left[\frac{1}{2}, 1\right]\) i.e including the Crank-Nicolson and backward Euler scheme.
	\item Finally, some numerical experiments are conducted to conclude our theoretical results. Moreover, we discuss the behavior of the state variable and control inputs for various parameters.
\end{itemize}

In this paper, we adopt the following notations \cite{MR1625845}:
\begin{itemize}
	\item All strongly measurable functions $ v:[0,T]\rightarrow X $ in the space $L^{p}((0,T);X)$ have the norm:
	\begin{align*}
		\norm{v}_{L^{p}((0,T);X)}=
		\begin{cases}
			\left(\int_{0}^{T}\norm{v}_{X}^{p}dt\right)^{\frac{1}{p}}< \infty, &\text{if} \ 1 \leq p<\infty, \\
			\mathop{\mathrm{ess\, sup}} \limits_{0\leq t\leq T} (\norm{v(t)}_{X})< \infty, &\text{if}\  p=\infty.
		\end{cases}
	\end{align*}
	Here, \( X \) is a Banach space equipped with the norm \( \| \cdot \|_{X} \). Denote by $(.,.)$ the $L^{2}$-inner product and let the corresponding norm is denoted by \( \| \cdot \| \). For simplicity of use, we represent $ L^{p}((0,T);X)$ as $ L^{p}(X).$ Moreover, we use the norm  \(\norm{|v|}:=\sqrt{\norm{v}^{2}+v^{2}(0)+v^{2}(1)}\), which is equivalent to the $H^{1}$-norm.
	
	\item The Sobolev space is defined as 
	\[
	H^{m}(0,1)=\{v \mid v\in L^{2}(0,1),\ \frac{\partial^{j} v}{\partial x^{j}}\in L^{2}(0,1), \text{ for } j = 1, 2, \ldots, m \}.
	\]
	\item \textbf{Young's Inequality:} For all $ c,d >0$ and $\epsilon>0$,
	\begin{align}
		cd\leq \epsilon\frac{c^{p}}{p}+\frac{1}{\epsilon^{\frac{q}{p}}}\frac{d^{q}}{q},
	\end{align}
	where $ 1<p,q<\infty $ and $ \frac{1}{p} +\frac{1}{q} =1.$
	\item \textbf{Poincar\'e-Wirtinger's Inequality:} For \(v\in H^{1}(0,1)\), there holds
		\[ \norm{v}^{2}\leq \norm{ v_{x}}^{2}+ 2v^{2}(i), \quad i=0,1.
		\]
	\item \textbf{Boundary Trace Embedding Theorem \cite{MR4139150} :} There exists a bounded linear map
	\[
	\gamma_{1}:H^{1}(\Omega)\rightarrow L^{q}(\partial\Omega) \quad \text{for} \quad 2\leq q\leq \infty,
	\] such that
	\[ \norm{\gamma_{1}(v)}_{L^{q}(\partial\Omega)}\leq C \norm{v}_{H^{1}(\Omega)}.
	\]
	\item \textbf{Friedrichs's Inequality:} For \(v\in H^{1}(\Omega)\), there holds
	\[ \norm{v}^{2}\leq C_{F}\left(\norm{\nabla v}^{2}+ \norm{v}^{2}_{L^{2}(\partial\Omega)}\right),
	\] where \(C_{F}\) is a positive constant.
	\item \textbf{Discrete Gronwall Inequality \cite{doi:10.1137/0727022}.} Let \(k, B\), and \(a_{j}, b_{j}, c_{j}, l_{j}\), for integers $j\geq0$, be non-negative numbers such that
	\begin{align*}
		a_{n}+ k\sum_{j=0}^{n}b_{j}\leq k\sum_{j=0}^{n}l_{j}a_{j}+k\sum_{j=0}^{n}c_{j}+ B \quad \text{for} \quad n\geq 0.
	\end{align*}
Assume that \(kl_{j}<1\), for all \(j\), and set \(\alpha_{j}=(1-kl_{j})^{-1}\). Then 
\begin{align*}
	a_{n}+ k\sum_{j=0}^{n}b_{j}\leq \exp\big( k\sum_{j=0}^{n}\alpha_{j}l_{j}\big) \Big(k\sum_{j=0}^{n}c_{j}+ B\Big) \quad \text{for} \quad n\geq 0.
\end{align*}
\end{itemize}

This paper is organized as follows. Section \(2\) presents the construction of a theta scheme and exponential decay bounds for the fully discrete solution in \(1D\) when \(\theta\in[\frac{1}{2}, 1].\) Section \(3\) discusses corresponding error analysis for the state variable and control inputs. Section \(4\) deals with the \(2D\) viscous Burgers' equation, constructing a theta scheme and estimating exponential bounds for the fully discrete scheme, along with existence and uniqueness results when \(\theta\in[\frac{1}{2}, 1].\) In Section \(5\), we analyze error estimates for the state variable and control input. Section \(6\) contains some numerical experiments to verify our theoretical results.

 A stationary solution of \eqref{Eq1}-\eqref{Eq2} with zero Neumann boundary conditions has a constant solution (say \(w_{d}\)).  For more details, see \cite{ALLEN20021165, MR3100771, MR1247468, MR1651458}.
 Let us assume that \(w_{d}\geq 0\) for simplicity.
 
 In order to show stabilization of this equation, we consider  $w=y-w_{d},$  such that
  \[ \lim_{t\to \infty}y(t,x)=w_{d}  \  \text{for all} \ \ x\in[0,1],\]
 and \(w=0, \ \text{as} \ t\rightarrow \infty.\)
 Then, from \eqref{Eq1}-\eqref{Eq2}, we can write the following system 
\begin{align}
\label{w11}
w_{t}-\nu w_{xx} + w_{d}w_{x}+ ww_{x} &= 0, \;\;\;\; (x,t)\in(0,1)\times(0,\infty) , \\
w_{x}(0,t) &= v_{0}(t),   \;\;\;\; t\in(0,\infty),  \\
w_{x}(1,t) &= v_{1}(t),  \;\; \;\; t\in(0,\infty),	 \\
w\left(x,0\right) &= y_{0}(x)-w_{d} =: w_{0}(x) \ (say),\;\; \; \;x\in(0,1),
\label{w12}
\end{align}
where $v_{0}(t)$ and $v_{1}(t)$ are  defined as feedback control law  \cite{MR3790146} of the form 
\begin{align*}
v_{0}(t)&= \frac{1}{\nu}\left( (c_{0}+w_{d}) w(0,t) + \frac{2}{9c_{0}}w^{3}(0,t)\right),\\
v_{1}(t)&= -\frac{1}{\nu}\left( (c_{1}+w_{d})w(1,t)+\frac{2}{9c_{1}}w^{3}(1,t)\right),
\end{align*}
with $ c_{0}, c_{1} $ are positive constants. These Neumann boundary feedback control  are invertible functions, so we can also  apply for the Dirichlet boundary control \cite{MR1751258}.

The weak formulation of the problem \eqref{w11}-\eqref{w12} is to find $ w(t)\in H^{1}(0,1)$ such that 
	\begin{align}
	\label{1.9}
	(w_{t},\phi) +\nu (w_{x},\phi_{x})+w_{d}(w_{x},\phi) +(ww_{x},\phi)+ \left((c_{0}+w_{d})w(0,t)+  \frac{2}{9c_{0}}w^{3}(0,t)\right)\phi(0)\\ \nonumber+\left((c_{1}+w_{d})w(1,t)+ \frac{2}{9c_{1}}w^{3}(1,t)\right)\phi(1)=0, \hspace{5mm} \forall \  \phi \in H^{1},
	\end{align}
with $ w(x,0)=w_{0}(x).$

The following theorem is useful in the error analysis proof for the state variable and control inputs.
\begin{theorem}
	\label{L2.12}
	Suppose that \(w_{0}\in H^{3}(0,1)\). Then there exists \(0<\alpha \leq \frac{1}{2}\min\{\nu, (c_{0}+w_{d}), (c_{1}+w_{d})\}\) and a positive constant \(C\) such that
	\begin{align*}
		\norm{w}_{2}^{2}&+\norm{w_{t}}_{1}^{2}+\min\{1, \frac{1}{\nu}\}E_{1}(t)+E_{3}(t)+ e^{-2\alpha t}\int_{0}^{t}e^{2\alpha s}\norm{w_{txx}(s)}^{2}ds\\&+e^{-2\alpha t}\int_{0}^{t}e^{2\alpha s}\Big(\beta \norm{|{w(s)}|}^{2}+ \nu \norm{ w_{xx}(s)}^{2}+ E_{2}(s)+\norm{w_{t}(s)}_{1}^{2}\Big)ds
		\\&+e^{-2\alpha t} \int_{0}^{t}e^{2\alpha s}\Big(\norm{w_{tt}(s)}_{1}^{2}+\norm{w_{ttt}(s)}^{2}\Big)ds+\norm{w_{tt}}_{1}^{2} 
		\\& \leq C(\norm{w_{0}}_{1}, \norm{w_{0}}_{3}) e^{-2\alpha t} \exp(C(\norm{w_{0}}_{2}^{2}+ \norm{w_{0}}_{1}^{4})), 
	\end{align*}
	where \[\beta=\min\{ 2(\nu-\alpha), (c_{0}+w_{d}-2\alpha), (c_{1}+w_{d}-2\alpha)\}>0,\]
	\[E_{1}(t)=\Big((c_{0}+w_{d})w^{2}(0,t)+(c_{1}+w_{d})w^{2}(1,t)+\frac{1}{9c_{0}}w^{4}(0,t)+\frac{1}{9c_{1}}w^{4}(1,t)\Big),
	\]
	\[E_{2}(t)=\Big((c_{0}+w_{d})w_{t}^{2}(0,t)+(2c_{1}+3w_{d})w_{t}^{2}(1,t)+\frac{4}{3c_{0}}w^{2}(0,t)w_{t}^{2}(0,t)+\frac{4}{3c_{1}}w^{2}(1,t)w_{t}^{2}(1,t)\Big),
	\] and 
	\[E_{3}(t)=\Big((c_{0}+w_{d})w_{t}^{2}(0,t)+(c_{1}+w_{d})w_{t}^{2}(1,t)+\frac{2}{3c_{0}}w^{2}(0,t)w_{t}^{2}(0,t)+\frac{2}{3c_{1}}w^{2}(1,t)w_{t}^{2}(1,t)\Big).
	\]
	 
\end{theorem}
\begin{proof}
 For the proof of the estimate $L^{\infty}(H^{2})$, we refer the reader to \cite{MR3790146}. The proof for the remaining parts is given in the appendix.
\end{proof}

\section{Finite Element Analysis.}
	For any positive integer $ N $, let $ I=\{0=x_{0}<x_{1}<\cdots<x_{N}=1 \} $ be a partition of $ [0,1] $ into sub-intervals(finite elements) $ I_{j}=(x_{j-1},x_{j}) , \hspace{3mm} 1 \leq j \leq N,$ with length $ h_{j}=x_{j}-x_{j-1} $. Let the mesh parameter $ h=\max_{1 \leq j \leq N}h_{j}.$ We construct a finite dimensional subspace $ V_{h} $ of $ H^{1} $ as $$V_{h}=\{v_{h}\in C^{0}([0,1]):v_{h}|_{I_{j}} \in P_{1}(I_{j}), 1 \leq j \leq N \},$$ where \(P_{1}\) is the polynomial of degree at most one on each element \(I_{j}, \ j=1, 2, \ldots, N.\)
	
	The weak form of semi-discrete scheme for the problem \eqref{w11}-\eqref{w12} is to find $ w_{h}(t)\in V_{h}$ such that
	\begin{align}
	\label{2.1}
	(w_{ht},\phi) +\nu (w_{hx},\phi_{x})&+w_{d}(w_{hx},\phi) +(w_{h}w_{hx},\phi)+ \left((c_{0}+w_{d})w_{h}(0,t)+  \frac{2}{9c_{0}}w_{h}^{3}(0,t)\right)\phi(0)\nonumber\\&+\left((c_{1}+w_{d})w_{h}(1,t)+ \frac{2}{9c_{1}}w_{h}^{3}(1,t)\right)\phi(1)=0, \quad \forall \ \phi \in V_{h}.
\end{align}
We introduce an auxiliary projection \(\tilde{w}_{h}(t)\in V_{h}\) of \(w(t)\) as follows
\begin{align*}
	\left( w_{x}(t)-\tilde{w}_{hx}(t), \phi_{x}\right)+ \lambda \left( w(t)-\tilde{w}_{h}(t), \phi \right)=0, \quad \forall \ \phi\in V_{h},
\end{align*}
where $\lambda$ is a positive constant. The existence and uniqueness for \(\tilde{w}_{h}(t)\) follows from the Lax-Milgram lemma for given \(w(t)\).
Set \(\eta := w-\tilde{w}_{h}(t)\). The following estimates hold:
\begin{align}
	\label{2.21}
	\begin{cases}
		\norm{\eta(t)}_{i}\leq Ch^{2-i}\norm{w}_{2}, \quad \norm{\eta_{t}(t)}_{i}\leq Ch^{2-i}\norm{w_{t}}_{2},  \quad i=0,1,
		\\
		|{\eta(x,t)}|\leq Ch^{2}\norm{w}_{2}, \quad |{\eta_{t}(x,t)}|\leq Ch^{2}\norm{w_{t}}_{2}, \quad x=0,1.
	\end{cases}
\end{align}
For  the proof of above estimates see \cite{MR3790146, thomee2007galerkin}.

The next theorem represents an error analysis of the semi-discrete scheme \eqref{2.1} for the state variable and control inputs.
\begin{theorem}
	\label{th2.11}
	Let \(w_{0}\in H^{3}(0,1).\) Then there exists a positive constant \(C\) independent of \(h\) such that
	\begin{align*}
		\norm{w-w_{h}}_{i}\leq C(\norm{w_{0}}_{3})h^{2-i}e^{-\alpha t}\exp^{(C\norm{w_{0}}_{2})}, \quad i=0,1,
	\end{align*}
and 
\begin{align*}
	\norm{v_{j}(t)-v_{jh}(t)}\leq C(\norm{w_{0}}_{3})h^{2}e^{-\alpha t}\exp^{(C\norm{w_{0}}_{2})}, \quad j=0,1,
\end{align*}
where \(0<\alpha \leq\frac{1}{2}\min \{\nu, (c_{0}+\dfrac{w_{d}}{2}-\frac{\nu}{2}), (c_{1}+ w_{d}-\frac{\nu}{2})\}\)
and 
\begin{align*}
	v_{0h}(t)&= \frac{1}{\nu}\left( (c_{0}+w_{d}) w_{h}(0,t) + \frac{2}{9c_{0}}w_{h}(0,t)^{3}\right),\\
	v_{1h}(t)&= -\frac{1}{\nu}\left( (c_{1}+w_{d})w_{h}(1,t)+\frac{2}{9c_{1}}w_{h}(1,t)^{3}\right).
\end{align*}
\end{theorem}
\begin{remark}
	The proof of Theorem \ref{th2.11} is provided in \cite{MR3790146}.
\end{remark}
\subsection{Fully Discrete Scheme for \(1D\) Viscous Burgers' Equation.}
This section provides the construction of a theta scheme and some {\it{a priori}} estimates for the state variable when $\theta\in[\frac{1}{2}, 1]$,  which helps in the proof of error analysis for the state variable and control inputs.

	We construct a theta scheme for the time discretization to the semidiscrete scheme \eqref{2.1}.
	 For a smooth function $ \phi $ defined on $ [0,T],$ set $ \phi ^{n}=\phi(t_{n}) $, and $ {\delta_{t}^{+}}\phi^{n} =\frac{\phi^{n+1}-\phi ^{n}}{k}.$ Denote $ W^{n}$ as the fully discrete approximation of  $w(t_{n})$ and \(t_{n+\theta}:=\theta t_{n}+ (1-\theta)t_{n-1}\), \ $ \theta \in [0,1], $ where \(t_{n}=nk, \ n=0, 1, \ldots, M\) and \(k=\frac{T}{M}\) is the time step size.
	
	Applying a theta scheme to the semi-discrete scheme \eqref{2.1}, we find a sequence $ \{W^{n}\}_{n\geq 1}$ such that
	\begin{align}
		\label{2.2}
	\nonumber	({\delta_{t}^{+}}W^{n},\phi) &+\nu (W_{x}^{n+\theta},\phi_{x})+w_{d}(W_{x}^{n+\theta},\phi) +(W^{n+\theta}W_{x}^{n+\theta},\phi)+ (c_{0}+w_{d})W^{n+\theta}(0)\phi(0)\\& +  \frac{2}{9c_{0}}(W^{n+\theta})^{3}(0)\phi(0)+\left((c_{1}+w_{d})W^{n+\theta}(1)+ \frac{2}{9c_{1}}(W^{n+\theta})^{3}(1)\right)\phi(1)=0, \quad \forall \phi \in V_{h},
	\end{align}
	where $  W^{n+\theta}:= \theta W^{n+1}+ (1-\theta) W^{n}.$	
	 \begin{remark}
		The fully discrete scheme \eqref{2.2} is known as forward Euler method, backward Euler method, and Crank-Nicolson scheme when $ \theta =0, \theta =1,$ and $ \theta =\frac{1}{2},$ respectively. 
	\end{remark}
	In the following lemma, we discuss the exponential bound of the approximate solution $ \{W^{n}\}_{n\geq 1} $ for a fully discrete scheme \eqref{2.2}.
	\begin{lemma}
		\label{L2.1}
		Let $ W^{0}\in H^{1}(0,1) $ and $ 0\leq\alpha \leq \frac{\theta^{2}}{2}\min\{\nu, \frac{c_{0}+w_{d}}{2},\frac{c_{1}+w_{d}}{2}\}$, where $\theta\in[\frac{1}{2}, 1].$  Suppose that \(k_{0}>0\) such that for \(0<k\leq k_{0}\)
		\begin{align}
			\label{2.41}
			e^{2\alpha k}\leq1+k\theta^{2}\min\{\nu , \frac{(c_{i}+w_{d})}{2} \},\quad i=0,1.
		\end{align}
		 Then, the following holds
		\begin{align*}
			\norm{W^{M}}^{2}+ke^{-2\alpha t_{M}}\beta_{1}\sum_{n=0}^{M-1}\bigg( \norm{\hat{W}_{x}^{n+1}}^{2}&+\sum_{i=0}^{1}\Big((\hat{W}^{n+1})^{2}(i)+e^{2\alpha t_{n}} ({W}^{n+1})^{4}(i)\Big) \bigg)
			\\&\leq Ce^{-2\alpha t_{M}}\norm{W^{0}}_{1}^{2},
		\end{align*}
	where \(\beta_{1}\) is given in \eqref{2.8}  below.
	\end{lemma}
\begin{proof}
	Set $ \phi = W^{n+\theta} $ in \eqref{2.2} to obtain
	\begin{align}
	\label{2.3}
	\nonumber	({\delta_{t}^{+}}W^{n}, W^{n+\theta})&+\nu \norm{W_{x}^{n+\theta}}^{2}+ \frac{w_{d}}{2}\left((W^{n+\theta})^{2}(1)-(W^{n+\theta})^{2}(0)\right)\nonumber\\&+ (c_{0}+w_{d})(W^{n+\theta})^{2}(0)+  \frac{2}{9c_{0}}(W^{n+\theta})^{4}(0)+(c_{1}+w_{d})(W^{n+\theta})^{2}(1)\nonumber\\&+ \frac{2}{9c_{1}}(W^{n+\theta})^{4}(1)=\frac{1}{3}\left( (W^{n+\theta})^{3}(0)-(W^{n+\theta})^{3}(1)\right).
	\end{align}
	Using Young's inequality to the right hand side of \eqref{2.3} yields
	\begin{align*}
		\frac{1}{3}(W^{n+\theta})^{3}(i)\leq \frac{c_{i}}{2}(W^{n+\theta})^{2}(i)+\frac{1}{18c_{i}}(W^{n+\theta})^{4}(i),\ i=0,1.
		%\frac{1}{3}(W^{n+\theta})^{3}(1)\leq \frac{c_{1}}{2}(W^{n+\theta})^{2}(1)+\frac{1}{18c_{1}}(W^{n+\theta})^{4}(1).
	\end{align*}
	Note that \[({\delta_{t}^{+}}W^{n}, W^{n+\theta})=\frac{1}{2}{\delta_{t}^{+}}\norm{W^{n}}^{2}+k(\theta-\frac{1}{2})\norm{{\delta_{t}^{+}}W^{n}}^{2},
	\]
 therefore from \eqref{2.3}, we arrive at
	\begin{align}
		\label{2.42}
		\frac{1}{2}{\delta_{t}^{+}}\norm{W^{n}}^{2}+k(\theta-\frac{1}{2})\norm{{\delta_{t}^{+}}W^{n}}^{2}&+\nu \norm{W_{x}^{n+\theta}}^{2}+\frac{1}{2}\left((c_{0}+w_{d})(W^{n+\theta})^{2}(0)+  \frac{1}{3c_{0}}(W^{n+\theta})^{4}(0)\right)\nonumber \\&+\frac{1}{2}\left((c_{1}+3w_{d})(W^{n+\theta})^{2}(1)+ \frac{1}{3c_{1}}(W^{n+\theta})^{4}(1)\right)\leq 0.
	\end{align}
 Since  $ \theta \geq \frac{1}{2} $, it follows that
	\begin{align*}
		\frac{1}{2}{\delta_{t}^{+}}\norm{W^{n}}^{2}+\nu \norm{W_{x}^{n+\theta}}^{2}&+\frac{1}{2}\left((c_{0}+w_{d})(W^{n+\theta})^{2}(0)+  \frac{1}{3c_{0}}(W^{n+\theta})^{4}(0)\right)\\&+\frac{1}{2}\left((c_{1}+3w_{d})(W^{n+\theta})^{2}(1)+ \frac{1}{3c_{1}}(W^{n+\theta})^{4}(1)\right)\leq 0.
	\end{align*}
Multiplying by \(2e^{2\alpha t_{n+1}}\) and applying
\begin{align}
	\label{exp}
e^{\alpha t_{n+1}}{\delta_{t}^{+}}W^{n}=e^{\alpha k}{\delta_{t}^{+}}\hat{W}^{n}-\frac{(e^{\alpha k}-1)}{k}\hat{W}^{n+1},
\end{align} 
 gives
\begin{align}
	\label{2.6}
	e^{2\alpha k}{\delta_{t}^{+}}\norm{\hat{W}^{n}}^{2}&-\frac{(e^{2\alpha k}-1)}{k}\norm{\hat{W}^{n+1}}^{2} +2\nu e^{2\alpha t_{n+1}} \norm{W_{x}^{n+\theta}}^{2}+e^{2\alpha t_{n+1}}\left((c_{0}+w_{d})(W^{n+\theta})^{2}(0)\right)
	\nonumber\\&+  \frac{e^{2\alpha t_{n+1}}}{3c_{0}}(W^{n+\theta})^{4}(0)+e^{2\alpha t_{n+1}}\left((c_{1}+3w_{d})(W^{n+\theta})^{2}(1)+ \frac{1}{3c_{1}}(W^{n+\theta})^{4}(1)\right)\leq 0,
\end{align}
where \(\hat{W}^{n}:=e^{\alpha t_{n}}W^{n} \). 

Note that 
\[\norm{W_{x}^{n+\theta}}^{2}=\theta^{2}\norm{W_{x}^{n+1}}^{2}+(1-\theta)^{2}\norm{W_{x}^{n}}^{2}+2\theta(1-\theta)(W_{x}^{n+1}, W_{x}^{n}), 
\]
and \[(W^{n+\theta})^{2}(i)=\theta^{2}(W^{n+1})^{2}(i)+(1-\theta)^{2}(W^{n})^{2}(i)+2\theta(1-\theta)W^{n+1}(i)W^{n}(i), \quad i=0,1.
\]
Also
\begin{align*}
	(W^{n+\theta})^{4}(i)&=\Big(\theta^{2}(W^{n+1})^{2}(i)+ (1-\theta)^{2}(W^{n})^{2}(i)+2\theta (1-\theta)(W^{n+1})(i)(W^{n})(i)\Big)^{2},
	\\&=\theta^{4}(W^{n+1})^{4}(i)+ (1-\theta)^{4}(W^{n})^{4}(i)+6\theta^{2} (1-\theta)^{2}(W^{n+1})^{2}(i)(W^{n})^{2}(i)
	\\&\quad + 4\theta^{3} (1-\theta)(W^{n+1})^{3}(i)(W^{n})(i)+4\theta (1-\theta)^{3}(W^{n+1})(i)(W^{n})^{3}(i), \quad i=0,1.
\end{align*}
Substituting these estimates into \eqref{2.6} and using  Young's inequality in the resulting inequality, we arrive at 
\begin{align}
	\label{2.7}
	e^{2\alpha k}{\delta_{t}^{+}}\norm{\hat{W}^{n}}^{2}&-\frac{(e^{2\alpha k}-1)}{k}\norm{\hat{W}^{n+1}}^{2} +\nu \theta^{2}\norm{\hat{W}_{x}^{n+1}}^{2}+\sum_{i=0}^{1}\left(\frac{(c_{i}+w_{d})}{2}\theta^{2}(\hat{W}^{n+1})^{2}(i)\right)
	\nonumber\\&+ \sum_{i=0}^{1}\frac{\theta^{4}e^{2\alpha t_{n+1}}}{6c_{i}}({W}^{n+1})^{4}(i)
	\nonumber\\& \leq (1-\theta)^{2}e^{2\alpha t_{n+1}}\Big(2\nu \norm{W^{n}_{x}}^{2}+ \sum_{i=0}^{1}(c_{i}+w_{d})W^{n}(i)+\sum_{i=0}^{1}C(1-\theta)^{2}(W^{n})^{4}(i)\Big).
\end{align}
Using Poincar\'e-Wirtinger's inequality
\[\norm{\hat{W}^{n+1}}^{2}\leq \norm{\hat{W}_{x}^{n+1}}^{2}+(\hat{W}^{n+1})^{2}(0)+(\hat{W}^{n+1})^{2}(1),
\] and multiplying by \(e^{-2\alpha k}\) in the resulting inequality, it follows from \eqref{2.7}
\begin{align*}
	{\delta_{t}^{+}}\norm{\hat{W}^{n}}^{2}&+\Big( \nu \theta^{2}e^{-2\alpha k}-\frac{(1-e^{-2\alpha k})}{k}\Big)\norm{\hat{W}_{x}^{n+1}}^{2}+\sum_{i=0}^{1}\left(\frac{(c_{i}+w_{d})}{2}\theta^{2}e^{-2\alpha k}-\frac{(1-e^{-2\alpha k})}{k}\right)(\hat{W}^{n+1})^{2}(i)
\nonumber\\&+ \sum_{i=0}^{1}\Big(\frac{\theta^{4}}{6c_{i}}e^{2\alpha t_{n}}\Big)({W}^{n+1})^{4}(i)
	\nonumber\\& \leq (1-\theta)^{2}e^{2\alpha t_{n}}\Big(2\nu \norm{W^{n}_{x}}^{2}+ \sum_{i=0}^{1}(c_{i}+w_{d})W^{n}(i)+\sum_{i=0}^{1}C(1-\theta)^{2}(W^{n})^{4}(i)\Big).
\end{align*}
Multiplying by \(k\) and summing over \(n=0\) to \(M-1\) yields
\begin{align*}
	\norm{\hat{W}^{M}}^{2}+k\beta_{1}\sum_{n=0}^{M-1}\bigg( \norm{\hat{W}_{x}^{n+1}}^{2}&+\sum_{i=0}^{1}\Big((\hat{W}^{n+1})^{2}(i)+ e^{2\alpha t_{n}}({W}^{n+1})^{4}(i)\Big) \bigg)
	\\&\leq \norm{W^{0}}^{2}+C\bigg(\norm{W^{0}}_{1}^{2}+\sum_{i=0}^{1}\Big((W^{0})^{2}(i)+(W^{0})^{4}(i)\Big) \bigg),
\end{align*}
where \begin{align}
	\label{2.8}
\beta_{1}=\min\{\Big(\nu \theta^{2}e^{-2\alpha k}-\frac{(1-e^{-2\alpha k})}{k}\Big), \Big(\frac{(c_{i}+w_{d})}{2}\theta^{2}e^{-2\alpha k}-\frac{(1-e^{-2\alpha k})}{k}\Big), \frac{\theta^{4}}{6c_{i}}\}, \quad i=0,1.
\end{align}
Choose \(k_{0}>0\) such that \eqref{2.41} meets  for \(0<k\leq k_{0}\).

This proof is completed after multiplying by \(e^{-2\alpha t_{M}}\).
\end{proof}
In the following lemma, we discuss the {\it{a priori}} exponential bounds of a fully  discrete scheme in the \(H^{1}\) and \(L^{\infty}\)-norms.
\begin{lemma}
	\label{L2.3}
	Assume that \(W^{0}\in H^{1}(0,1) \) and \( \theta \in[\frac{1}{2}, 1]\). Then, the following  holds
	\begin{align*}
			\nu \norm{W_{x}^{M}}^{2}&+ \sum_{i=0}^{1}\Big((c_{i}+w_{d})\big(W^{M}\big)^{2}(i)+\frac{1}{18c_{i}} \big(W^{M}\big)^{4}(i)\Big)
			\\&+ke^{-2\alpha t_{M}}\sum_{n=0}^{M-1}e^{2\alpha t_{n}}\norm{{\delta_{t}^{+}}W^{n}}^{2}\leq Ce^{-2\alpha t_{M}}\norm{W^{0}}_{1}^{2}e^{C\norm{W^{0}}^{2}},
	\end{align*}
and 
\begin{align*}
	\norm{W^{n}}_{\infty}^{2}\leq Ce^{-2\alpha t_{M}}\norm{W^{0}}_{1}^{2}e^{C\norm{W^{0}}^{2}}, \quad 0\leq n\leq M.
\end{align*}
\end{lemma}
\begin{proof}
	Setting \(\phi={\delta_{t}^{+}}W^{n}\) in \eqref{2.2} yields
	\begin{align}
		\label{2.12}
		\nu ({\delta_{t}^{+}}W_{x}^{n}, W_{x}^{n+\theta})+ \norm{{\delta_{t}^{+}}W^{n}}^{2}&+ (c_{0}+w_{d})W^{n+\theta}{\delta_{t}^{+}}W^{n}(0)+(c_{1}+w_{d})W^{n+\theta}{\delta_{t}^{+}}W^{n}(1)
		\nonumber\\& +\frac{2}{9c_{0}}(W^{n+\theta})^{3}(0){\delta_{t}^{+}}W^{n}(0)+\frac{2}{9c_{1}}(W^{n+\theta})^{3}(1){\delta_{t}^{+}}W^{n}(1)
	\nonumber\\&=-w_{d}(W_{x}^{n+\theta}, {\delta_{t}^{+}}W^{n})-(W^{n+\theta}W_{x}^{n+\theta}, {\delta_{t}^{+}}W^{n}).
	\end{align}
Note that
\begin{align*}
	\delta_{t}^{+}W^{n}(0) \Big(W^{n+\theta}\Big)^{3}(0)&=k^{3}\Big(\theta-\frac{1}{2}\Big)^{3}(\delta_{t}^{+}W^{n})^{4}(0)+ \frac{3k^{2}}{2}\Big(\theta-\frac{1}{2}\Big)^{2}\big(\delta_{t}^{+}W^{n}\big)^{3}(0)\Big(W^{n+1}(0)+W^{n}(0)\Big)
	\\& \quad +\frac{3k}{4}\big(\theta-\frac{1}{2}\big)(\delta_{t}^{+}W^{n})^{2}(0)\Big(W^{n+1}(0)+W^{n}(0)\Big)^{2}
	\\& \quad+\frac{1}{8}\delta_{t}^{+}W^{n}(0)\Big(W^{n+1}(0)+W^{n}(0)\Big)^{3},
\end{align*}
where \begin{align*}
W^{n+\theta}(0)&=\big(\theta-\frac{1}{2}\big)\Big(W^{n+1}(0)-W^{n}(0)\Big)+ \frac{1}{2}\Big(W^{n+1}(0)+W^{n}(0)\Big),
\\&=k\big(\theta-\frac{1}{2}\big)\delta_{t}^{+}W^{n}(0)+ \frac{1}{2}\Big(W^{n+1}(0)+W^{n}(0)\Big).
\end{align*}

Also 
\begin{align*}
	\frac{1}{8}\delta_{t}^{+}W^{n}(0)\Big(W^{n+1}(0)+W^{n}(0)\Big)^{3}&=\frac{1}{8}\delta_{t}^{+}\big(W^{n}\big)^{4}(0)+\frac{3}{8k}\Big(W^{n}(0)\big(W^{n+1}\big)^{3}(0)-W^{n+1}(0)\big(W^{n}\big)^{3}(0)\Big),
	\\& = \frac{1}{8}\delta_{t}^{+}(W^{n})^{4}(0)+\frac{3}{8}\delta_{t}^{+}(W^{n})^{2}(0)W^{n+1}(0)W^{n}(0).
\end{align*}
Using the Cauchy-Sachwarz  inequality and Young's inequality, we deduce from \eqref{2.12}
\begin{align*}
	\nu ({\delta_{t}^{+}}W_{x}^{n}, W_{x}^{n+\theta})+\frac{1}{2} \norm{{\delta_{t}^{+}}W^{n}}^{2}&+(c_{0}+w_{d})W^{n+\theta}{\delta_{t}^{+}}W^{n}(0)+(c_{1}+w_{d})W^{n+\theta}{\delta_{t}^{+}}W^{n}(1)
	\nonumber\\& +\sum_{i=0}^{1}\frac{k^{3}}{18c_{i}}(\theta-\frac{1}{2})^{3}({\delta_{t}^{+}}W^{n})^{4}(i)+\frac{1}{36c_{i}}{\delta_{t}^{+}}(W^{n})^{4}(i)
	\nonumber\\&\leq \frac{C}{k}\norm{W^{n}}_{\infty}^{2}\Big((W^{n+1})^{2}(0)+(W^{n+1})^{2}(1)\Big)
	\\&\quad + \frac{C}{k}\sum_{i=0}^{1}\Big((W^{n+1})^{4}(i)+(W^{n})^{4}(i)\Big)+C\norm{W^{n+\theta}}_{1}^{2}\\&\quad +C\norm{W_{x}^{n+\theta}}^{2}\norm{W^{n+\theta}}_{1}^{2}.
\end{align*}
Note that
\[ ({\delta_{t}^{+}}W_{x}^{n}, W_{x}^{n+\theta})=\frac{1}{2}{\delta_{t}^{+}}\norm{W_{x}^{n}}^{2}+k(\theta-\frac{1}{2})\norm{{\delta_{t}^{+}}W_{x}^{n}}^{2},
\] and \[W^{n+\theta}{\delta_{t}^{+}}W^{n}(i)=\frac{1}{2}{\delta_{t}^{+}}(W^{n})^{2}(i)+k(\theta-\frac{1}{2})\big({\delta_{t}^{+}}W^{n}\big)^{2}(i), \quad i=0,1.
\]
Since \(\theta\geq \frac{1}{2}\), we arrive at
\begin{align*}
	\nu {\delta_{t}^{+}}\norm{W_{x}^{n}}^{2}+ \sum_{i=0}^{1}\Big((c_{i}+w_{d}){\delta_{t}^{+}}\big(W^{n}\big)^{2}(i)&+ \frac{1}{18c_{i}}{\delta_{t}^{+}}\big(W^{n}\big)^{4}(i)\Big)+ \norm{{\delta_{t}^{+}}W^{n}}^{2}
	\\&\leq \frac{C}{k}\norm{W^{n}}_{1}^{2}\Big((W^{n+1})^{2}(0)+(W^{n+1})^{2}(1)\Big)
	\\&\quad + \frac{C}{k}\sum_{i=0}^{1}\Big((W^{n+1})^{4}(i)+(W^{n})^{4}(i)\Big)+C\norm{W^{n+\theta}}_{1}^{2}\\&\quad +C\norm{W_{x}^{n+\theta}}^{2}\norm{W^{n+\theta}}_{1}^{2}.
\end{align*}
Multiplying by \(e^{2\alpha t_{n+1}}\) to the above inequality and using \eqref{exp}, it follows that after that multiplying by $e^{-2\alpha k}$ in the resulting inequality
\begin{align*}
	\nu {\delta_{t}^{+}}\norm{\hat{W}_{x}^{n}}^{2}+ \sum_{i=0}^{1}\Big((c_{i}+w_{d}){\delta_{t}^{+}}\big(\hat{W}^{n}\big)^{2}(i)&+ \frac{e^{2\alpha t_{n+1}}}{18c_{i}}{\delta_{t}^{+}}\big(W^{n}\big)^{4}(i)\Big)+ e^{2\alpha t_{n}}\norm{{\delta_{t}^{+}}W^{n}}^{2}
    \\&\leq C(\alpha)\Big( \norm{\hat{W}_{x}^{n+1}}^{2}+\sum_{i=0}^{1}\big(\hat{W}^{n+1}\big)^{2}(i)+ e^{2\alpha t_{n}}\big(W^{n+1}\big)^{4}(i)  \Big)
	\\&\quad+ Ce^{2\alpha t_{n}}\norm{W^{n}}_{1}^{2}\Big((W^{n+1})^{2}(0)+(W^{n+1})^{2}(1)\Big)
	\\&\quad + Ce^{2\alpha t_{n}}\sum_{i=0}^{1}\Big((W^{n+1})^{4}(i)+(W^{n})^{4}(i)\Big)+Cke^{2\alpha t_{n}}\norm{W^{n+\theta}}_{1}^{2}\\&\quad +Cke^{2\alpha t_{n}}\norm{W_{x}^{n+\theta}}^{2}\norm{W^{n+\theta}}_{1}^{2}.
\end{align*}
 We choose \(k\) sufficiently small such that \((\nu-Ck)>0 \). Summing from \(n=0\) to \(M-1\) and using discrete Gronwall's inequality with Lemma \ref{L2.1}.
We complete the proof of the first part after multiplying by $e^{-2\alpha t_{M}}$.

Since
\[ 	\norm{W^{n}}_{\infty}^{2}\leq 2\Big(\norm{W_{x}^{n}}^{2}+ (W^{n})^{2}(0)+(W^{n})^{2}(1) \Big), \quad 0\leq n\leq M,
\]
 therefore using the first part, the proof is completed.
\end{proof}
\subsection{Existence and Uniqueness.}
 This subsection provides the existence and uniqueness of a fully discrete scheme \eqref{2.2} with the help of the Brouwer's fixed point theorem.
 \begin{theorem}\label{th2.1}(\textbf{Brouwer's fixed point theorem}.) Let $ H $ be finite dimensional Hilbert space with inner product $(.,.)$ and $ \norm{.} $. Let $g:H \rightarrow H$ be a continuous function. If there exist $ \alpha>0 $ such that $ (g(x),x)\geq 0, \ \ \forall x$ with $ \norm{x}=\alpha,$ then there exist $ x^{*} $ in $ H $ such that $ \norm{x^{*}}\leq \alpha $ and $ g(x^{*})=0.$	
 \end{theorem}

 \begin{lemma}
 	\label{L2.2}
 	Given $ W^{n}, $ the solution of the discrete scheme \eqref{2.2} exists.
 \end{lemma}
\begin{proof}
	We consider $ H $ as $ V_{h},$ and denoted $ V:=W^{n+\theta}.$ Define a function $ g:H\rightarrow H $ such that \begin{align}
	\label{2.4}
		 (g(V) ,\phi)&=\frac{1}{k\theta}(V- W^{n}, \phi)+\nu(V_{x}, \phi_{x})+w_{d}(V_{x},\phi)  +(V V_{x},\phi)\nonumber\\&\quad+ \left((c_{0}+w_{d})V(0)+  \frac{2}{9c_{0}}(V)^{3}(0)\right)\phi(0) +\left((c_{1}+w_{d})V(1)+ \frac{2}{9c_{1}}(V)^{3}(1)\right)\phi(1),
	\end{align}
	where ${\delta_{t}^{+}}W^{n}=\frac{1}{k \theta}(W^{n+\theta}-W^{n}),$ for $ \theta\neq0. $
	
	For \(\theta=0,\) the discrete scheme is explicit. Thus, we know the solution at each time level \(t_{n}.\)
	
	As $g$ is a polynomial function of $V$, therefore $g$ is a continuous function.
	
	Choose $ \phi =V $ in $ \eqref{2.4} $ to have
	\begin{align*}
		(g(V) ,V)&=\frac{1}{k\theta}\left(\norm{V}^{2}-(W^{n}, V)\right)+ \nu\norm{V_{x}}^{2}+ \frac{w_{d}}{2}\left(V^{2}(1)-V^{2}(0)\right) +\frac{1}{3}\left(V^{3}(1)-V^{3}(0)\right)\\&\quad+ \left((c_{0}+w_{d})V^{2}(0)+  \frac{2}{9c_{0}}(V)^{4}(0)\right) +\left((c_{1}+w_{d})V^{2}(1)+ \frac{2}{9c_{1}}(V)^{4}(1)\right).
	\end{align*}
	Using Young's inequality, we arrive at
	\begin{align*}
		(g(V), V)&\geq \frac{\norm{V}}{k\theta}\left(\norm{V}-\norm{W^{n}} \right)+ \nu\norm{V_{x}}^{2}
		 +\frac{1}{2}\left((c_{0}+w_{d})V^{2}(0)+  \frac{1}{3c_{0}}(V)^{4}(0)\right) \\&\quad +\frac{1}{2}\left((c_{1}+3w_{d})V^{2}(1)+ \frac{1}{3c_{1}}(V)^{4}(1)\right).
	\end{align*}
	Therefore, for $ \norm{V}= 1+\norm{W^{n}},$ we have $ (g(V), V)\geq0.$ Hence, Theorem \ref{th2.1} yields existence of $ V^{*}\in H$ such that $ g(V^{*})=0.$
\end{proof}
 Next lemma shows the uniqueness of the fully discrete scheme.
\begin{lemma}
	The approximate solution of the scheme $ \eqref{2.2} $ is unique.
\end{lemma}
\begin{proof}
	For the proof, we follow the proof of Lemma \ref{L21.3} below.
\end{proof}
	\section{Error analysis.}
	This section contains the error analysis of the state variable corresponding to a fully discrete scheme \eqref{2.2}.
	
Define an auxiliary projection \(\tilde{w}_{h}(t_{n})\in V_{h}\) of \(w(t_{n})\) through the following form for all \(n\geq0\)
\begin{align}
	\label{3.11}
        \left( w_{x}(t_{n})-\tilde{w}_{hx}(t_{n}), \phi_{x}\right)+ \lambda \left( w(t_{n})-\tilde{w}_{h}(t_{n}), \phi \right)=0, \quad \forall \ \phi\in V_{h}.
\end{align}

	Splitting the error $ e^{n}:=w^{n}-W^{n}=(w^{n}- \tilde{w}_{h}^{n})- (W^{n}-\tilde{w}_{h}^{n})=: \eta^{n}-\xi^{n},$ where \(\eta^{n}= (w^{n}- \tilde{w}_{h}^{n}), \) and \(\xi^{n}=( W^{n}-\tilde{w}_{h}^{n})\). Here,
	 $ w^{n}:=w(t_{n})$ is the solution of the problem \eqref{1.9} at $ t=t_{n},$ and $ W^{n} $ is the solution of a fully discrete scheme $ \eqref{2.2}.$

 Setting $ t=t_{n+\theta}$ in \eqref{1.9} and subtracting from \eqref{2.2}, we obtain with the use of \eqref{3.11}
 \begin{align}
 \label{3.1}
 	({\delta_{t}^{+}}\xi^{n}, \phi) +\nu (\xi_{x}^{n+\theta}, \phi_{x})&+w_{d}(\xi_{x}^{n+\theta}, \phi)+ (c_{0}+w_{d})\xi^{n+\theta}(0)\phi(0)+ (c_{1}+w_{d})\xi^{n+\theta}(1)\phi(1)\nonumber\\[1mm]&=  ({\delta_{t}^{+}}\eta^{n}- \nu \lambda \eta^{n+\theta}, \phi)+(w^{n+\theta}w_{x}^{n+\theta}-W_{x}^{n+\theta}W^{n+\theta}, \phi)
 	\nonumber\\&\quad+w_{d}(\eta_{x}^{n+\theta}, \phi)+ (c_{0}+w_{d})\eta^{n+\theta}(0)\phi(0)+ (c_{1}+w_{d})\eta^{n+\theta}(1)\phi(1) 
 	\nonumber\\&\quad +\frac{2}{9c_{0}}\left((w^{n+\theta})^{3}(0)-(W^{n+\theta})^{3}(0)\right)\phi(0)\nonumber\\[1mm]& \quad +\frac{2}{9c_{1}}\left((w^{n+\theta})^{3}(1)-(W^{n+\theta})^{3}(1)\right)\phi(1)+(w_{t}^{n+\theta}-{\delta_{t}^{+}}w^{n}, \phi),
 \end{align}
where 
\[ (w^{n})^{3}(i)-(W^{n})^{3}(i)= (\eta^{n})^{3}(i)- (\xi^{n})^{3}(i)+ 3w^{n}(i)\eta^{n}(i)\left(w^{n}(i)-\eta^{n}(i)\right)- 3W^{n}(i)\xi^{n}(i)\left(W^{n}(i)-\xi^{n}(i)\right), 
\]
 \(i=0, 1,\)
and \[w^{n}w_{x}^{n}-W_{x}^{n}W^{n}= w_{x}^{n}(\eta^{n}-\xi^{n})+ W^{n}(\eta_{x}^{n}-\xi_{x}^{n}).
\]
In the following lemma, we discuss the error analysis of the state variable corresponding to a theta scheme \eqref{2.2}.
\begin{lemma}
	\label{L3.1}
	Let \(w_{0}\in H^{3}(0,1)\) and $ 0\leq \alpha\leq \frac{\theta^{2} \beta_{2}}{4}$, where $\theta\in[\frac{1}{2},1]$. Assume that \(k_{0}>0\) such that for \(0<k\leq k_{0}\)
	\begin{align}
		\label{3.3}
		e^{2\alpha k}\leq 1+\frac{k\theta^{2}\beta_{2}}{2},
	\end{align} 
where \(\beta_{2}=\min\{ \nu, (c_{0}+\frac{w_{d}}{2}-\frac{\nu}{2}), (c_{1}+w_{d}-\frac{\nu}{2})\}>0\). 
Then there exists a positive constant \(C=C(\norm{w_{0}}_{3})\) independent of \(h\) and \(k\) such that
	\begin{align*}
		\norm{{\xi}^{M}}^{2}&+ ke^{-2\alpha t_{M}}\Big(\frac{\theta^{2}\beta_{2}}{2}e^{-2\alpha k}-\frac{(1-e^{-2\alpha k})}{k}\Big) \sum_{n=0}^{M-1}\Big(\norm{\hat{\xi}_{x}^{n+1}}^{2}+ (\hat{\xi}^{n+1})^{2}(0)+ (\hat{\xi}^{n+1})^{2}(1)\Big)
		\nonumber\\&+  ke^{-2\alpha t_{M}}\sum_{n=0}^{M-1}e^{2\alpha t_{n}}\sum_{i=0}^{1}\frac{\theta^{4}}{18c_{i}}(\xi^{n+1})^{4}(i)
		\leq Ce^{-2\alpha t_{M}}\Big(h^{4}+ k^{2}(\theta-\frac{1}{2})+k^{4}\Big).
	\end{align*} 
\end{lemma}
\begin{proof}
	Select \(\phi=\xi^{n+\theta}\) in \eqref{3.1} to have
	\begin{align}
		\label{3.31}
		({\delta_{t}^{+}}\xi^{n}, \xi^{n+\theta}) &+\nu \norm{\xi_{x}^{n+\theta}}^{2}+ (c_{0}+\frac{w_{d}}{2})(\xi^{n+\theta})^{2}(0)+ (c_{1}+\frac{3w_{d}}{2})(\xi^{n+\theta})^{2}(1)+\frac{2}{9c_{0}}(\xi^{n+\theta})^{4}(0)\nonumber\\&+\frac{2}{9c_{1}}(\xi^{n+\theta})^{4}(1)+\frac{2}{3c_{0}}(\xi^{n+\theta})^{2}(0)(W^{n+\theta})^{2}(0)+\frac{2}{3c_{1}}(\xi^{n+\theta})^{2}(1)(W^{n+\theta})^{2}(1)\nonumber\\[1mm]&=({\delta_{t}^{+}}\eta^{n}- \nu \lambda\eta^{n+\theta}, \xi^{n+\theta})+w_{d}(\eta_{x}^{n+\theta}, \xi^{n+\theta})+ \sum_{i=0}^{1}(c_{i}+w_{d})\eta^{n+\theta}(i)\xi^{n+\theta}(i)\nonumber\\&\quad + \Big( w_{x}^{n+\theta}(\eta^{n+\theta}-\xi^{n+\theta}), \xi^{n+\theta}\Big)+\Big( W^{n+\theta}(\eta_{x}^{n+\theta}-\xi_{x}^{n+\theta}), \xi^{n+\theta}\Big)\nonumber\\&\quad +\sum_{i=0}^{1}\frac{2}{9c_{i}}(\eta^{n+\theta})^{3}(i)\xi^{n+\theta}(i) +\sum_{i=0}^{1}\frac{2}{3c_{i}}\eta^{n+\theta}(i) w^{n+\theta}(i)\xi^{n+\theta}(i)\left( w^{n+\theta}(i)-\eta^{n+\theta}(i)\right)\nonumber\\&\quad +\sum_{i=0}^{1}\frac{2}{9c_{i}}(\xi^{n+\theta})^{3}(i)W^{n+\theta}(i) + (w_{t}^{n+\theta}-{\delta_{t}^{+}}w^{n}, \xi^{n+\theta}),
		\nonumber\\& =\sum_{i=1}^{9}I_{i}.
	\end{align}
The first term \(I_{1}\) on the right hand side of \eqref{3.31} is bounded by
\begin{align*}
 I_{1}=({\delta_{t}^{+}}\eta^{n}- \nu \lambda\eta^{n+\theta}, \xi^{n+\theta})\leq C\left(\norm{{\delta_{t}^{+}}\eta^{n}}^{2}+ \norm{\eta^{n+\theta}}^{2}\right)+ \frac{\nu}{16}\norm{\xi^{n+\theta}}^{2}.
\end{align*}
Applying Young's inequality, the second  \(I_{2}\) and third  \(I_{3}\) terms on the right hand side of \eqref{3.31} yield
\begin{align*}
	I_{2}= w_{d}(\eta_{x}^{n+\theta}, \xi^{n+\theta})&\leq \frac{\nu}{12}\norm{\xi_{x}^{n+\theta}}^{2}+ \frac{w_{d}}{4}(\xi^{n+\theta})^{2}(0)+ w_{d}(\xi^{n+\theta})^{2}(1)\\&\quad+C\left( (\eta^{n+\theta})^{2}(0)+(\eta^{n+\theta})^{2}(1)+\norm{\eta^{n+\theta}}^{2}\right),
\end{align*}
and 
\begin{align*}
	I_{3}&=(c_{i}+w_{d})\eta^{n+\theta}(i)\xi^{n+\theta}(i)\leq C(\eta^{n+\theta})^{2}(i)+ \frac{c_{i}}{8}(\xi^{n+\theta})^{2}(i), \quad i=0, 1.
\end{align*}
On the right hand side of \eqref{3.31}, the fourth term \(I_{4}\) gives
\begin{align*}
	I_{4}= \Big( w_{x}^{n+\theta}(\eta^{n+\theta}-\xi^{n+\theta}), \xi^{n+\theta}\Big)\leq C\norm{\eta^{n+\theta}}^{2}+ C\norm{w_{x}^{n+\theta}}_{\infty}^{2}\norm{\xi^{n+\theta}}^{2}+ \frac{\nu}{16}\norm{\xi^{n+\theta}}^{2}.
\end{align*}
Setting \(W^{n+\theta}=\tilde{w}_{h}^{n+\theta}+\xi^{n+\theta}\) in the first term of \(I_{5}\) on the right hand side, we obtain using Young's inequality
\begin{align*}
	I_{5}&=\Big((\tilde{w}_{h}^{n+\theta}+\xi^{n+\theta})\eta_{x}^{n+\theta}, \xi^{n+\theta}\Big)- ( W^{n+\theta}\xi_{x}^{n+\theta}, \xi^{n+\theta}),
	\\&\leq C\Big((\tilde{w}_{h}^{n+\theta})^{2}(0)(\eta^{n+\theta})^{2}(0)+(\tilde{w}_{h}^{n+\theta})^{2}(1)(\eta^{n+\theta})^{2}(1)+\big(1+\norm{\tilde{w}_{h}^{n+\theta}}^{2}\big) \norm{\eta^{n+\theta}}^{2} \Big)\\&\quad+ \frac{\nu}{6}\norm{\xi_{x}^{n+\theta}}^{2}+C\Big(\norm{\eta_{x}^{n+\theta}}^{2}+ \norm{ \tilde{w}_{hx}^{n+\theta}}_{\infty}^{2}+ \norm{ W^{n+\theta}}_{\infty}^{2}\Big)\norm{\xi^{n+\theta}}^{2}+ \frac{\nu}{16}\norm{\xi^{n+\theta}}^{2}\\&\quad+\frac{c_{0}}{8}(\xi^{n+\theta})^{2}(0)++\frac{c_{1}}{8}(\xi^{n+\theta})^{2}(1).
\end{align*}
Again, using the Young's inequality, the terms \(I_{6}\) and \(I_{7}\) gives
\begin{align*}
	I_{6}&=\frac{2}{9c_{i}}(\eta^{n+\theta})^{3}(i)\xi^{n+\theta}(i) \leq C(\eta^{n+\theta})^{6}(i)+ \frac{c_{i}}{8}(\xi^{n+\theta})^{2}(i),
\end{align*}
and
\begin{align*}
	\frac{2}{3c_{i}}\eta^{n+\theta}(i) w^{n+\theta}(i)\xi^{n+\theta}(i)\left( w^{n+\theta}(i)-\eta^{n+\theta}(i)\right)&\leq C\Big( (\eta^{n+\theta})^{2}(i)(w^{n+\theta})^{4}(i)+ (\eta^{n+\theta})^{4}(i)(w^{n+\theta})^{2}(i)\Big)\\&\quad+ \frac{c_{i}}{8}(\xi^{n+\theta})^{2}(i), \quad i=0,1.
\end{align*}
 From the second last term \(I_{8}\), we obtain
\begin{align*}
	\frac{2}{9c_{i}}(\xi^{n+\theta})^{3}(i)W^{n+\theta}(i)\leq \frac{6}{36c_{i}}(\xi^{n+\theta})^{4}(i)+ +\frac{2}{3c_{i}}(\xi^{n+\theta})^{2}(i)(W^{n+\theta})^{2}(i), \quad i=0, 1.
\end{align*}
An application of the Young's inequality the final term \(I_{9}\) yields
\begin{align*}
	I_{9}=(w_{t}^{n+\theta}-{\delta_{t}^{+}}w^{n}, \xi^{n+\theta})\leq C\norm{w_{t}^{n+\theta}-{\delta_{t}^{+}}w^{n}}^{2}+ \frac{\nu}{16}\norm{\xi^{n+\theta}}^{2}.
\end{align*}
Substituting \(I_{j}, j=1, 2\ldots, 9\) into \eqref{3.31}, using \(\norm{\tilde{w}_{hx}}_{\infty}\leq C\norm{w}_{2}\), and Lemma \ref{L2.1} and Theorem \ref{L2.12} with \(\alpha=0\), we arrive at
\begin{align*}
	({\delta_{t}^{+}}\xi^{n}, \xi^{n+\theta}) &+\frac{3\nu}{4} \norm{\xi_{x}^{n+\theta}}^{2}+ (\frac{c_{0}}{2}+\frac{w_{d}}{4})(\xi^{n+\theta})^{2}(0)+ (\frac{c_{1}}{2}+\frac{w_{d}}{2})(\xi^{n+\theta})^{2}(1)+\sum_{i=0}^{1}\frac{1}{18c_{i}}(\xi^{n+\theta})^{4}(i)\nonumber\\[1mm]&\leq C\Big(\norm{\eta_{x}^{n+\theta}}^{2}+ \norm{w^{n+\theta}}_{2}^{2}+ \norm{ W^{n+\theta}}_{1}^{2}\Big)\norm{\xi^{n+\theta}}^{2}+\frac{\nu}{4}\norm{\xi^{n+\theta}}^{2}\\&\quad+ C\Big(\norm{{\delta_{t}^{+}}\eta^{n}}^{2}+ \norm{\eta^{n+\theta}}^{2}+\sum_{i=0}^{1}\big( (\eta^{n+\theta})^{2}(i)+(\eta^{n+\theta})^{4}(i)+(\eta^{n+\theta})^{6}(i) \big) \Big)\\&\quad + C\norm{w_{t}^{n+\theta}-{\delta_{t}^{+}}w^{n}}^{2}.
\end{align*}
Applying Poincar\'e-Wirtinger's to the above inequality yields
\begin{align}
	\label{3.41}
	({\delta_{t}^{+}}\xi^{n}, \xi^{n+\theta}) +\frac{\nu}{2} \norm{\xi_{x}^{n+\theta}}^{2}&+ (\frac{c_{0}}{2}+\frac{w_{d}}{4}-\frac{\nu}{4})(\xi^{n+\theta})^{2}(0)+ (\frac{c_{1}}{2}+\frac{w_{d}}{2}-\frac{\nu}{4})(\xi^{n+\theta})^{2}(1)+\sum_{i=0}^{1}\frac{1}{18c_{i}}(\xi^{n+\theta})^{4}(i)\nonumber\\[1mm]&\leq C\Big(\norm{\eta_{x}^{n+\theta}}^{2}+ \norm{w^{n+\theta}}_{2}^{2}+ \norm{ W^{n+\theta}}_{1}^{2}\Big)\norm{\xi^{n+\theta}}^{2}+C\norm{w_{t}^{n+\theta}-{\delta_{t}^{+}}w^{n}}^{2}
	\nonumber\\&\quad+ C\Big(\norm{{\delta_{t}^{+}}\eta^{n}}^{2}+ \norm{\eta^{n+\theta}}^{2}+\sum_{i=0}^{1}\big( (\eta^{n+\theta})^{2}(i)+(\eta^{n+\theta})^{4}(i)+(\eta^{n+\theta})^{6}(i) \big) \Big).
\end{align}
Note that
\begin{align*}
	({\delta_{t}^{+}}{\xi}^{n}, \xi^{n+\theta})=\frac{1}{2}{\delta_{t}^{+}}\norm{{\xi}^{n}}^{2}+k(\theta- \frac{1}{2})\norm{{\delta_{t}^{+}}{\xi}^{n}}^{2}.
\end{align*}
From \eqref{3.41}, it follows that 
\begin{align}
	\label{3.51}
	{\delta_{t}^{+}}\norm{{\xi}^{n}}^{2}+2k(\theta- \frac{1}{2})\norm{{\delta_{t}^{+}}{\xi}^{n}}^{2} &+ \beta_{2}\Big(  \norm{\xi_{x}^{n+\theta}}^{2}+ (\xi^{n+\theta})^{2}(0)+ (\xi^{n+\theta})^{2}(1)\Big)+\sum_{i=0}^{1}\frac{1}{9c_{i}}(\xi^{n+\theta})^{4}(i)\nonumber\\[1mm]&\leq C\Big(\norm{\eta_{x}^{n+\theta}}^{2}+ \norm{w^{n+\theta}}_{2}^{2}+ \norm{ W^{n+\theta}}_{1}^{2}\Big)\norm{\xi^{n+\theta}}^{2}+C\norm{w_{t}^{n+\theta}-{\delta_{t}^{+}}w^{n}}^{2}
	\nonumber\\&\quad+ C\Big(\norm{{\delta_{t}^{+}}\eta^{n}}^{2}+ \norm{\eta^{n+\theta}}^{2}+\sum_{i=0}^{1}\big( (\eta^{n+\theta})^{2}(i)+(\eta^{n+\theta})^{4}(i)+(\eta^{n+\theta})^{6}(i) \big) \Big),
\end{align}
where \(\beta_{2}=\min\{ \nu, (c_{0}+\frac{w_{d}}{2}-\frac{\nu}{2}), (c_{1}+w_{d}-\frac{\nu}{2})\}>0\).

For \(\theta\geq \frac{1}{2}\),  we have from \eqref{3.51}
\begin{align}
	\label{3.61}
	{\delta_{t}^{+}}\norm{{\xi}^{n}}^{2}&+ \beta_{2}\Big(  \norm{\xi_{x}^{n+\theta}}^{2}+ (\xi^{n+\theta})^{2}(0)+ (\xi^{n+\theta})^{2}(1)\Big)+\sum_{i=0}^{1}\frac{1}{9c_{i}}(\xi^{n+\theta})^{4}(i)\nonumber\\[1mm]&\leq C\Big(\norm{\eta_{x}^{n+\theta}}^{2}+ \norm{w^{n+\theta}}_{2}^{2}+ \norm{ W^{n+\theta}}_{1}^{2}\Big)\norm{\xi^{n+\theta}}^{2}+C\norm{w_{t}^{n+\theta}-{\delta_{t}^{+}}w^{n}}^{2}
	\nonumber\\&\quad+ C\Big(\norm{{\delta_{t}^{+}}\eta^{n}}^{2}+ \norm{\eta^{n+\theta}}^{2}+\sum_{i=0}^{1}\big( (\eta^{n+\theta})^{2}(i)+(\eta^{n+\theta})^{4}(i)+(\eta^{n+\theta})^{6}(i) \big) \Big).
\end{align}
Using a Taylor series expansion of $ w^{n} $ about the point $ t_{n+\theta} $ in the interval $ (t_{n}, t_{n+1}),$ it follows that
\begin{align*}
	w_{t}^{n+\theta}-{\delta_{t}^{+}}w^{n}&= -\frac{((1-\theta)^{2}-\theta^{2})}{2k}\left(\int_{t_{n}}^{t_{n+\theta}}(t_{n+\theta}- s)w_{tt}(s)ds+\int_{t_{n+\theta}}^{t_{n+1}}(t_{n+1}- s)w_{tt}(s)ds\right),\\[1mm]& \quad
	+k\frac{((1-\theta)^{3}+\theta^{3})}{6}\left(\int_{t_{n}}^{t_{n+\theta}}(t_{n+\theta}- s)w_{ttt}(s)ds+\int_{t_{n+\theta}}^{t_{n+1}}(t_{n+1}- s)w_{ttt}(s)ds\right).
\end{align*}
% 0r
% \begin{align*}
	% {\delta_{t}^{+}}w_{h}^{n}-w_{ht}^{n+\theta}= \frac{k((1-\theta)^{2}-\theta^{2})}{2}w_{htt}^{n+\theta}+\frac{k^{2}((1-\theta)^{3}+\theta^{3})}{6}w_{httt}^{n+\theta}(\eta),
	% \end{align*}
% where $ \eta \in [t_{n+1}, t_{n}].$
Hence,  with the help of H\"older inequality, we get
\begin{align}
	\label{3.71}
	\nonumber\norm{w_{t}^{n+\theta}-{\delta_{t}^{+}}w^{n}}^{2}&\leq \frac{k\theta^{3}\left((1-\theta)^{2}-\theta^{2}\right)^{2}}{3}\int_{t_{n}}^{t_{n+1}}\norm{w_{tt}(s)}^{2}ds+\frac{k^{3}\theta^{3}\left((1-\theta)^{3}+\theta^{3}\right)^{2}}{12}\int_{t_{n}}^{t_{n+1}}\norm{w_{ttt}(s)}^{2}ds\\[1mm]&=:T^{n+\theta}.
\end{align}
From \eqref{3.61} and \eqref{3.71}, we arrive at
\begin{align}
	\label{3.8}
	{\delta_{t}^{+}}\norm{{\xi}^{n}}^{2}&+ \beta_{2}\Big(  \norm{\xi_{x}^{n+\theta}}^{2}+ (\xi^{n+\theta})^{2}(0)+ (\xi^{n+\theta})^{2}(1)\Big)+\sum_{i=0}^{1}\frac{1}{9c_{i}}(\xi^{n+\theta})^{4}(i)\nonumber\\[1mm]&\leq C\Big(\norm{\eta_{x}^{n+\theta}}^{2}+ \norm{w^{n+\theta}}_{2}^{2}+ \norm{ W^{n+\theta}}_{1}^{2}\Big)\norm{\xi^{n+\theta}}^{2}+CT^{n+\theta}
	\nonumber\\&\quad+ C\Big(\norm{{\delta_{t}^{+}}\eta^{n}}^{2}+ \norm{\eta^{n+\theta}}^{2}+\sum_{i=0}^{1}\big( (\eta^{n+\theta})^{2}(i)+(\eta^{n+\theta})^{4}(i)+(\eta^{n+\theta})^{6}(i) \big) \Big).
\end{align}
Multiplying \eqref{3.8} by \(e^{2\alpha t_{n+1}}\) and following the proof of Lemma \ref{L2.1} with \eqref{exp},  we get from \eqref{3.8} by applying Poincar\'e-Wirtinger's inequality
\begin{align*}
	{\delta_{t}^{+}}\norm{\hat{\xi}^{n}}^{2}&+\Big(\frac{\theta^{2}\beta_{2}e^{-2\alpha k}}{2}-\frac{(1-e^{-2\alpha k})}{k}\Big)\Big(\norm{\hat{\xi}_{x}^{n+1}}^{2}+ (\hat{\xi}^{n+1})^{2}(0)+ (\hat{\xi}^{n+1})^{2}(1)\Big)+e^{2\alpha t_{n}}\sum_{i=0}^{1}\frac{\theta^{4}}{18c_{i}}(\xi^{n+1})^{4}(i)
	\nonumber\\[1mm]&\leq (1-\theta)^{2}e^{2\alpha t_{n}}\beta_{2}\Big(  \norm{\xi_{x}^{n}}^{2}+ (\xi^{n})^{2}(0)+ (\xi^{n})^{2}(1)\Big)+e^{2\alpha t_{n}}\sum_{i=0}^{1}C(1-\theta)^{4}(\xi^{n})^{4}(i)
		\nonumber\\[1mm]&\quad +Ce^{2\alpha t_{n}}\Big(\norm{\eta_{x}^{n+\theta}}^{2}+ \norm{w^{n+\theta}}_{2}^{2}+ \norm{ W^{n+\theta}}_{1}^{2}\Big)\norm{\xi^{n+\theta}}^{2}+Ce^{2\alpha t_{n}}T^{n+\theta}
		\nonumber\\&\quad+ Ce^{2\alpha t_{n}}\Big(\norm{{\delta_{t}^{+}}\eta^{n}}^{2}+ \norm{\eta^{n+\theta}}^{2}+\sum_{i=0}^{1}\big( (\eta^{n+\theta})^{2}(i)+(\eta^{n+\theta})^{4}(i)+(\eta^{n+\theta})^{6}(i) \big) \Big).
\end{align*}

Multiplying by \(k\) and summing from \(n=0\) to \(M-1\), we observe that
\begin{align*}
\norm{{\hat{\xi}}^{M}}^{2}&+ k\Big(\frac{\theta^{2}\beta_{2}}{2}e^{-2\alpha k}-\frac{(1-e^{-2\alpha k})}{k}\Big) \sum_{n=0}^{M-1}\Big(\norm{\hat{\xi}_{x}^{n+1}}^{2}+ (\hat{\xi}^{n+1})^{2}(0)+ (\hat{\xi}^{n+1})^{2}(1)\Big)
\nonumber\\&+  k\sum_{n=0}^{M-1}e^{2\alpha t_{n}}\sum_{i=0}^{1}\frac{\theta^{4}}{18c_{i}}(\xi^{n+1})^{4}(i)
\\&\leq Ck \sum_{n=0}^{M-1}e^{2\alpha t_{n}}\Big(\norm{\eta_{x}^{n+\theta}}^{2}+ \norm{w^{n+\theta}}_{2}^{2}+ \norm{ W^{n+\theta}}_{1}^{2}\Big)\norm{\xi^{n+\theta}}^{2}+Ck \sum_{n=0}^{M-1} e^{2\alpha t_{n}}T^{n+\theta}
\nonumber\\&\quad+ Ck\sum_{n=0}^{M-1}e^{2\alpha t_{n}}\Big(\norm{{\delta_{t}^{+}}\eta^{n}}^{2}+ \norm{\eta^{n+\theta}}^{2}+\sum_{i=0}^{1}\big( (\eta^{n+\theta})^{2}(i)+(\eta^{n+\theta})^{4}(i)+(\eta^{n+\theta})^{6}(i) \big) \Big),
\end{align*}
where \(\xi^{0}=0\).
Choose \(k_{0}>0\) such that \eqref{3.3}  is satisfied  for \(0<k\leq k_{0}\).

Note that \[\norm{{\delta_{t}^{+}}\eta^{n}}^{2}\leq \frac{1}{k}\int_{t_{n}}^{t_{n+1}}\norm{\eta_{t}}^{2}ds.
\]
We choose \(k\) to be sufficiently small such that \(\big(1-Ck\theta^{2}\big)>0\). Using discrete Gronwall's inequality, Theorem \ref{L2.12} with \(\alpha=0\), and Lemma \ref{L2.1}  with estimates \eqref{2.21}. This completes the proof after multiplying by $e^{-2\alpha t_{M}}$ in the resulting inequality.
\end{proof}

\begin{theorem}
	\label{th3.1}
	Suppose that the hypothesis of Lemma \ref{L3.1} holds. Then, there exists a positive constant \(C\) independent of \(h\) and \(k\) such that
	\begin{align*}
	 \norm{w^{n}-W^{n}}\leq C e^{-2\alpha t_{M}}\Big(h^{2}+ k(\theta-\frac{1}{2})+k^{2}\Big),
	\end{align*}
and 
\begin{align*}
	|v_{i}^{n}-V_{i}^{n}|\leq C e^{-2\alpha t_{M}}\Big(h^{2}+ k(\theta-\frac{1}{2})+k^{2}\Big), \quad i=0,1,
\end{align*}
where \(v_{i}^{n}:=v_{i}(t_{n})\) and \(V_{i}^{n}\) is given by 
\[V_{i}^{n}=(c_{i}+w_{d})W_{i}^{n}+ \frac{2}{9c_{i}}(W_{i}^{n})^{3}, \quad i=0,1.
\]
\end{theorem}
\begin{proof}
	The  first part follows from the estimate \eqref{2.21} and Lemmas \ref{L3.1} with use of triangle inequality.
	The proof of second part is followed by
	\begin{align*}
	  v_{i}^{n}-V_{i}^{n}= (c_{i}+w_{d})(\eta^{n}-\xi^{n})+\frac{2}{9c_{i}}(\eta^{n}-\xi^{n})\big((w^{n})^{2}+ (W^{n})^{2}+w^{n}W^{n}\big).
	\end{align*}
Hence, using Theorem \ref{L2.12} with $\alpha=0$, Lemmas \ref{L2.3} and \ref{L3.1} with the estimates \eqref{2.21}, we have
\begin{align*}
	|v_{i}^{n}-V_{i}^{n}|\leq C e^{-2\alpha t_{M}}(h^{2}+ k\left((1-\theta)^{2}-\theta^{2}\right)+k^{2}\left((1-\theta)^{3}+\theta^{3}\right)), \quad i=0,1.
\end{align*}
 This completes the proof.
\end{proof}

 	\section{2D Viscous Burgers' equation.}
 	In this section, we discuss a fully discrete scheme for the \(2D\) viscous Burgers' equation with nonlinear Neumann boundary feedback control. The existence and uniqueness of a fully discrete scheme is established.
 	 Moreover, error analysis of a theta scheme is discussed for the state variable and control input.

 We consider the following $ 2D$ viscous Burgers' equation with nonlinear Neumann boundary feedback control problem 
 \begin{align}
 	y_{t}-\nu \Delta y+ y(\nabla y\cdot \mathbf{1})&=0, \quad (x,t)\in \Omega \times (0,\infty),\label{11.1}\\
 	\frac{\partial y}{\partial n}(x,t)&=v_{2}(x,t), \quad (x,t)\in \partial\Omega \times (0,\infty),\label{11.2}\\
 	y(x,0)&=y_{0}(x), \quad x\in \Omega, \label{11.3}
 \end{align}
 where  $\nu>0$ is diffusion constant, $ v_{2}(x,t) $ is a scalar control input, $ \mathbf{1}=(1,1) ,$ $  y(\nabla y\cdot \mathbf{1})=y\big(\sum_{i=1}^{2}y_{x_{i}}\big) $ is the nonlinear term and $ y_{0}(x) $ is a given function and let $ \Omega \subset \mathbb{R}^{2} $ be a bounded domain with smooth boundary $\partial\Omega$. 
 The steady state problem corresponding to \eqref{11.1}-\eqref{11.3} is 
 \begin{align}
 	\nu \Delta y^{\infty}-y^{\infty}(\nabla y^{\infty}\cdot \mathbf{1})&=0, \quad x\in \Omega, \label{11.4}\\
 	\frac{\partial y^{\infty}(x)}{\partial n}&=0, \quad \text{on} \ \partial\Omega \label{11.5}.
 \end{align}
 Any constant is the solution of steady state problem \eqref{11.4}-\eqref{11.5} (say $ w_{d} $) (for more details, see \cite{MR4139150}). Let us assume that $ w_{d}\geq 0$ for simplicity. Our goal is to find 
 \begin{align}
 	\lim\limits_{t\rightarrow\infty}y(x,t)=w_{d}, \quad \forall x\in \Omega.
 \end{align}
   We consider $ w=y-w_{d} $ such that $ \lim\limits_{t\rightarrow \infty}w=0.$  Then, $ w $ satisfies the following form
 \begin{align}
 	w_{t}-\nu \Delta w+ w_{d}(\nabla w\cdot \mathbf{1})+ w(\nabla w\cdot \mathbf{1})&=0, \quad (x,t)\in \Omega \times (0,\infty),\label{11.7}\\
 	\frac{\partial w}{\partial n}(x,t)&=v_{2}(x,t) \quad (x,t)\in \partial\Omega \times (0,\infty),\label{11.8}\\
 	w(x,0)&=y_{0}-w_{d}=w_{0}(x)(\text{say}), \quad x\in \Omega\label{11.9},
 \end{align}
where \(v_{2}\) is defined as the boundary feedback control law  \cite{MR4139150} in the following form
 \[v_{2}(x,t)=-\frac{1}{\nu}\Big(2(c_{2}+w_{d})w+ \frac{2}{9c_{2}}w^{3}\Big),
\]
and \(c_{2}\) is a positive constant.

The weak formulation of \eqref{11.7}-\eqref{11.9} seeks \(w\in H^{1}(\Omega)\) such that
\begin{align}
	\label{weak}
	\nonumber (w_{t},\chi)+\nu (\nabla w,\nabla \chi)+ w_{d}(\nabla w\cdot \mathbf{1},\chi) +&(w(\nabla w\cdot \mathbf{1}),\chi)\\&
	+\int_{\partial\Omega}\left(2(c_{2}+w_{d})w+\frac{2}{9c_{2}}w^{3}\right) \chi d\Gamma=0, \quad \chi \in H^{1}.
\end{align}
For the proof of the following lemma see \cite{MR4139150}.
\begin{lemma}
	\label{L4.1}
Let \(w_{0}\in H^{3}(\Omega)\). Then there exists a decay \(0\leq\alpha\leq\frac{1}{C_{F}} \min\{\nu, (c_{2}+w_{d})\}\) and a positive constant \(C\) such that
\begin{align*}
	\norm{w}_{2}^{2}+\norm{w_{t}}_{1}^{2}&+2(c_{2}+w_{d})\Big(\norm{w}_{L^{2}(\partial\Omega)}^{2}+\norm{w_{t}}_{L^{2}(\partial\Omega)}\Big)+\min\{\frac{1}{9c_{2}}, \frac{2}{3c_{2}}\}\Big(\norm{w}_{L^{4}(\partial\Omega)}^{4}+ \norm{w_{t}}_{L^{4}(\partial\Omega)}^{4}\Big)
	\\&\quad+\gamma e^{-2\alpha t}\int_{0}^{t}e^{2\alpha s}\Big(\norm{\nabla w(s)}^{2}+\norm{w(s)}_{L^{2}(\partial\Omega)}^{2} +\frac{1}{3\gamma c_{2}}\norm{w(s)}_{L^{4}(\partial\Omega)}^{4}\Big)ds
	\\&\quad+C_{1}e^{-2\alpha t}\int_{0}^{t}e^{2\alpha s}\Big( \norm{\Delta w(s)}^{2}+ \norm{w_{t}(s)}_{2}^{2}+\norm{w(s)w_{t}(s)}_{L^{2}(\partial\Omega)}^{2}+ \norm{w_{t}(s)}_{L^{2}(\partial\Omega)}^{2}\Big)ds
	\\&\quad\leq C(\norm{w_{0}}_{3})e^{C(\norm{w_{0}}_{2})}e^{-2\alpha t}.
\end{align*}
where \(0<\gamma=2\min\{v-\alpha C_{F}, c_{2}+w_{d}-\alpha C_{F}\} \), \(C_{1}=\min\{\nu,2(c_{2}+w_{d}),\frac{2}{3c_{2}}\}>0\) and \(C_{F}>0\) is the Friedrich's inequality  constant.
\end{lemma}
\begin{remark}
	As demonstrated in \(1D\) case,  we have established bounds for \(\norm{w_{tt}}\) and \(\norm{w_{ttt}}\). Equivalently, similar bounds can be derived for \(\norm{w_{tt}}\) and \(\norm{w_{ttt}}\) in \(2D\) case.
\end{remark}
\subsection{Finite element method.}
This subsection contains the semi-discrete scheme of \eqref{weak} and error estimate of the semi-discrete scheme for the state variable and control input.

Define a finite dimensional subspace \(V_{h}\) of \(H^{1}\) such that \[ V_{h}=\{\phi_{h}\in C^{0}(\bar{\Omega}): \phi_{h}|_{K}\in P_{1}(K) \ \forall \  K\in T_{h}\},\]
where \(T_{h}\) is the regular triangulation into finitely many triangles with pairwise disjoint interiors and union of the triangles determine a polygonal domain \(\Omega_{h}\subset \Omega\) with boundary vertices on \(\partial \Omega\). Let \(h=\max_{K\in T_{h}}h_{K}\), where \(h_{K}=diam (K) \ \forall \  K\in T_{h}\).

 We consider the following semi-discrete approximation of  problem \eqref{11.7}-\eqref{11.9} seeks $ w_{h}\in V_{h} $ such that
 \begin{align}
 	\label{11.10}
 	\nonumber (w_{ht},\chi)+\nu (\nabla w_{h},\nabla \chi)+ w_{d}(\nabla w_{h}\cdot \mathbf{1},\chi) +&(w_{h}(\nabla w_{h}\cdot \mathbf{1}),\chi)\\&
 	+\int_{\partial\Omega}\left(2(c_{2}+w_{d})w_{h}+\frac{2}{9c_{2}}w_{h}^{3}\right)\chi d\Gamma=0, \quad \chi \in V_{h},
 \end{align}
 where $ w_{h}(0)=P_{h}y_{0}-w_{d}=w_{0h} $(say), an approximation of $ w_{0}$, and $ P_{h}y_{0} $ is the $ H^{1} $ projection of $ y_{0} $ onto $ V_{h} $ such that
 \begin{align}
 	\norm{y_{0}-y_{0h}}_{j}\leq Ch^{2-j}\norm{y_{0}}, \quad j=0,1.
 \end{align}
For more details see \cite{MR4139150}.

The following theorem contains the error estimate of the semi-discrete scheme \eqref{11.10} for the state variable and control input.
\begin{theorem}
	\label{th4.1}
	Let \(w_{0}\in H^{3}(\Omega)\). Then exists a positive constant \(C\) independent of \(h\) such that
	\begin{align*}
		\norm{w-w_{h}}_{L^{\infty}(H^{i})}\leq C(\norm{w_{0}}_{3}) e^{-\alpha t}h^{2-i} \exp(\norm{w_{0}}_{2}), \quad i=0,1
	\end{align*} and 
\begin{align*}
	\norm{v_{2}(t)-v_{2h}(t)}_{L^{\infty}(L^{2}(\partial \Omega))}\leq C(\norm{w_{0}}_{3}) e^{-\alpha t}h^{\frac{3}{2}} \exp(\norm{w_{0}}_{2}),
\end{align*}
where \(0\leq \alpha \leq \frac{1}{C_{F}} \min\{ \frac{3\nu}{4}, (\frac{c_{2}}{2}+w_{d})\}\) and \( v_{2h}(t)=-\frac{1}{\nu}\left(2(c_{2}+w_{d})w_{h}+\frac{2}{9c_{2}}w_{h}^{3}\right)\). Denote $H^{0}(\Omega)=L^{2}(\Omega).$
\end{theorem}
\begin{proof}
	For the proof see \cite{MR4139150}.
\end{proof}
 \subsection{Fully discrete scheme. }
 In this subsection, we discuss the formulation of a fully discrete scheme of the problem \eqref{11.7}-\eqref{11.9} using a theta scheme, where $\theta\in[\frac{1}{2}, 1]$.
 
 Applying a theta scheme to the semi-discrete scheme \eqref{11.10}  finds a sequence $ \{W^{n}\}_{n\geq 1} $ such that
 \begin{align}
 	\label{21.1}
 	\nonumber ({\delta_{t}^{+}}W^{n},\chi)+\nu (\nabla W^{n+\theta},\nabla \chi)&+ w_{d}(\nabla W^{n+\theta}\cdot \mathbf{1},\chi) +(W^{n+\theta}(\nabla W^{n+\theta}\cdot \mathbf{1}),\chi)\\&
 	+\int_{\partial\Omega}\left(2(c_{2}+w_{d})W^{n+\theta}+\frac{2}{9c_{2}}(W^{n+\theta})^{3}\right)\chi d\Gamma=0, \quad \chi \in V_{h}.
 \end{align}
 %where $ W^{n+\theta} =\theta W^{n+1}+ (1-\theta)W^{n},$ and $ \theta \in [0,1]. $

 In the following lemma, we derive an \it{{a priori}} bounds for the fully discrete scheme \eqref{21.1}.
 \begin{lemma} 
 	\label{L21.1}
 	Let $ W_{0}\in H^{1}(\Omega)$ and $ 0\leq \alpha\leq \frac{\theta^{2}}{2C_{F}}\min\{\nu, (c_{2}+w_{d})\}$, where $\theta\in[\frac{1}{2},1]$. Assume that \(k_{0}>0\) such that for \(0<k\leq k_{0}\)
 	\begin{align}
 		\label{5.2}
 		e^{2\alpha k}\leq 1+\frac{k\theta^{2}}{C_{F}}\min\{\nu,(c_{2}+w_{d}) \}.
 	\end{align}
 	 Then, the following holds:
 	\begin{align*}
 		\norm{W^{M}}^{2}+ ke^{-2\alpha t_{M}}\beta_{3}\sum_{n=0}^{M-1} \Big(\norm{\nabla \hat{W}^{n+1}}^{2}+\norm{\hat{W}^{n+1}}_{L^{2}(\partial \Omega)}^{2}\Big)&+\frac{k\theta^{4}}{6c_{2}}e^{-2\alpha t_{M}}\sum_{n=0}^{M-1} e^{2\alpha t_{n}}\norm{{W}^{n+1}}_{L^{4}(\partial \Omega)}^{4}
 		\\&\leq Ce^{-2\alpha t_{M}} \norm{ W^{0}}_{1}^{2},
 	\end{align*}
 where \(\beta_{3}\) is given in \eqref{5.4} below.
 \end{lemma}
\begin{proof}
 Choose $ \chi ={W^{n+\theta}} ,$ in \eqref{21.1} to obtain
 \begin{align}
 	\label{21.2}
 	\nonumber ({\delta_{t}^{+}}W^{n},W^{n+\theta})+\nu \norm{\nabla W^{n+\theta}}^{2}&+\int_{\partial\Omega}\left(2(c_{2}+w_{d})W^{n+\theta}+\frac{2}{9c_{2}}(W^{n+\theta})^{3}\right)W^{n+\theta} d\Gamma\\[1mm]&\quad= - w_{d}(\nabla W^{n+\theta}\cdot \mathbf{1},W^{n+\theta}) - \Big(W^{n+\theta}(\nabla W^{n+\theta}\cdot \mathbf{1}),W^{n+\theta}\Big).
 \end{align}
 
 The first term on the right hand side of \eqref{21.2} is bounded by
 \begin{align*}
 	-w_{d}(\nabla W^{n+\theta}\cdot \mathbf{1},W^{n+\theta})&=-\frac{w_{d}}{2}\int_{\Omega}\left((W^{n+\theta})_{x_{1}}^{2}+ (W^{n+\theta})_{x_{2}}^{2}\right)dX =-\frac{w_{d}}{2}\sum_{j=1}^{2}\int_{\partial\Omega}(W_{j}^{n+\theta})^{2}\cdot n_{j} d\Gamma,\\&\leq \frac{w_{d}}{\sqrt{2}}\int_{\partial\Omega}(W^{n+\theta})^{2} d\Gamma\leq w_{d}\int_{\partial\Omega}(W^{n+\theta})^{2} d\Gamma.
 \end{align*}
 Using Young's Inequality, the second term on the right hand side of \eqref{21.2} gives
 \begin{align*}
 	\Big(W^{n+\theta}(\nabla W^{n+\theta}\cdot \mathbf{1}),W^{n+\theta}\Big)\leq \frac{1}{3}\sum_{j=1}^{2}\int_{\partial\Omega}(W_{j}^{n+\theta})^{3}\cdot n_{j} d\Gamma\leq c_{2}\int_{\partial\Omega}(W^{n+\theta})^{2}d\Gamma + \frac{1}{18c_{2}}\int_{\partial\Omega}(W^{n+\theta})^{4}d\Gamma.
 \end{align*}
 Note that
 \begin{align*}
 	({\delta_{t}^{+}}W^{n},W^{n+\theta})=\frac{1}{2}{\delta_{t}^{+}}\norm{W^{n}}^{2}+ k(\theta-\frac{1}{2})\norm{ {\delta_{t}^{+}}W^{n}}^{2}.
 \end{align*}
  From \eqref{21.2}, we arrive at
 \begin{align}
 	\label{21.3}
 	\frac{1}{2}{\delta_{t}^{+}}\norm{W^{n}}^{2}+ k(\theta-\frac{1}{2})\norm{ {\delta_{t}^{+}}W^{n}}^{2}+ \nu \norm{\nabla W^{n+\theta}}^{2}&+\int_{\partial\Omega}\left((c_{2}+w_{d})(W^{n+\theta})^{2}+\frac{1}{6c_{2}}(W^{n+\theta})^{4}\right) d\Gamma\leq 0.
 \end{align}
 Therefore for $ \theta\geq\frac{1}{2},$  we can write \eqref{21.3} as
 \begin{align*}
 	{\delta_{t}^{+}}\norm{W^{n}}^{2}+ 2\nu \norm{\nabla W^{n+\theta}}^{2}&+2\int_{\partial\Omega}\left((c_{2}+w_{d})(W^{n+\theta})^{2}+\frac{1}{6c_{2}}(W^{n+\theta})^{4}\right) d\Gamma\leq 0.
 \end{align*}
Multiplying by \(e^{2\alpha t_{n+1}}\) and following the proof of Lemma \ref{L2.1}, we arrive at
\begin{align*}
	e^{2\alpha k}{\delta_{t}^{+}}\norm{\hat{W}^{n}}^{2}&-\Big(\frac{e^{2\alpha k}-1}{k}\Big)\norm{\hat{W}^{n+1}}^{2}+\nu \theta^{2} \norm{\nabla \hat{W}^{n+1}}^{2}+ (c_{2}+w_{d})\theta^{2}\norm{\hat{W}^{n+1}}_{L^{2}(\partial \Omega)}^{2}+\frac{\theta^{4}}{6c_{2}}e^{2\alpha t_{n+1}}\norm{{W}^{n+1}}_{L^{4}(\partial \Omega)}^{4}
	\\&\leq (1-\theta)^{2}e^{2\alpha t_{n+1}} \Big(2\nu \norm{\nabla W^{n}}^{2}+2 (c_{2}+w_{d})\norm{W^{n}}_{L^{2}(\partial \Omega)}^{2}+C(1-\theta)^{2}\norm{W^{n}}_{L^{4}(\partial \Omega)}^{4}\Big).
\end{align*}
By Friedrichs's inequality and multiplying by \(e^{-2\alpha k}\) in the resulting inequality, it follows that
\begin{align*}
	{\delta_{t}^{+}}\norm{\hat{W}^{n}}^{2}&+\Big(\nu \theta^{2}e^{-2\alpha k}-C_{F}\Big(\frac{1-e^{-2\alpha k}}{k}\Big) \Big)\norm{\nabla \hat{W}^{n+1}}^{2}+\Big((c_{2}+w_{d})\theta^{2}e^{-2\alpha k}- C_{F}\Big(\frac{1-e^{-2\alpha k}}{k}\Big)\Big)\norm{\hat{W}^{n+1}}_{L^{2}(\partial \Omega)}^{2}
	\\&+\frac{\theta^{4}}{6c_{2}}e^{2\alpha t_{n}}\norm{{W}^{n+1}}_{L^{4}(\partial \Omega)}^{4}
	\\&\leq (1-\theta)^{2}e^{2\alpha t_{n}} \Big(2\nu \norm{\nabla W^{n}}^{2}+2 (c_{2}+w_{d})\norm{W^{n}}_{L^{2}(\partial \Omega)}^{2}+C(1-\theta)^{2}\norm{W^{n}}_{L^{4}(\partial \Omega)}^{4}\Big).
\end{align*}
 Multiplying by $ k $ and summing from $ n=0 $ to $ M-1 $, we obtain
 \begin{align*}
 	\norm{\hat{W}^{M}}^{2}+ k\beta_{3}\sum_{n=0}^{M-1} \Big(\norm{\nabla \hat{W}^{n+1}}^{2}+\norm{\hat{W}^{n+1}}_{L^{2}(\partial \Omega)}^{2}\Big)&+\frac{k\theta^{4}}{6c_{2}}\sum_{n=0}^{M-1}e^{2\alpha t_{n}}\norm{{W}^{n+1}}_{L^{4}(\partial \Omega)}^{4}
 	\\&\leq C\Big( \norm{\nabla W^{0}}^{2}+ \norm{W^{0}}_{L^{2}(\partial \Omega)}^{2}+\norm{W^{0}}_{L^{4}(\partial \Omega)}^{4} \Big).
 \end{align*}
where \begin{align}
	\label{5.4}
	0<\beta_{3}=\min\{\nu \theta^{2}e^{-2\alpha k}-C_{F}\Big(\frac{1-e^{-2\alpha k}}{k}\Big), (c_{2}+w_{d})\theta^{2}e^{-2\alpha k}- C_{F}\Big(\frac{1-e^{-2\alpha k}}{k}\Big)\}.
\end{align}
Choose \(k_{0}>0\) such that \eqref{5.2} is satisfied for \(0<k\leq k_{0}\). The proof is completed after multiplying by \(e^{-2\alpha t_{M}}\).
\end{proof}
 \subsection{Existence and Uniqueness of Theta Scheme.}
 This subsection discusses the existence and uniqueness of the fully discrete scheme \eqref{21.1}.
 \begin{lemma}
 	\label{L21.2}
 	Given $ W^{n}, $ then, there exists a solution of the discrete scheme \eqref{21.1}.
 \end{lemma}
 \begin{proof} 
 	Given $ W^{n} $ and denote $ U:=W^{n+\theta},$  where $ W^{n+\theta}=\theta W^{n+1}+(1-\theta)W^{n}$. Define a function $ f : V_{h}\rightarrow V_{h}$ by 
 	\begin{align}
 		\label{21.4}
 		\nonumber	(f(U), \chi)=\frac{1}{\theta}(U-W^{n}, \chi)&+ k\nu (\nabla U,\nabla \chi)+ kw_{d}(\nabla U\cdot \mathbf{1},\chi)+k\Big(U(\nabla U\cdot \mathbf{1}),\chi\Big)\\&
 		+k\int_{\partial\Omega}\left(2(c_{2}+w_{d})U +\frac{2}{9c_{2}}(U)^{3}\right)\chi d\Gamma, \ \ \chi \in S_{h}, \ \text{and}  \quad \theta \neq 0.
 	\end{align}
 	Since $ f $ is a polynomial function of $ U, $ therefore $ f $ is continuous.
 	
  Set $ \chi = U $ in \eqref{21.4} to get
 	\begin{align*}
 		(f(U), U)=\frac{1}{\theta}(U-W^{n}, U)&+ k\nu \norm{\nabla U}^{2}+ kw_{d}(\nabla U\cdot \mathbf{1}, U)+k(U(\nabla U\cdot \mathbf{1}), U)\\&
 		+k\int_{\partial\Omega}\left(2(c_{2}+w_{d})U^{2} +\frac{2}{9c_{2}}(U)^{4}\right) d\Gamma.
 	\end{align*}
 	Using Young's inequality to the above equation yields
 	\begin{align*}
 		(f(U), U)\geq \frac{1}{\theta}\norm{U}(\norm{U}-\norm{W^{n}})+k\nu \norm{\nabla U}^{2}+k\int_{\partial\Omega}\left(3(c_{2}+w_{d})U^{2} +\frac{5}{9c_{2}}(U)^{4}\right) d\Gamma.
 	\end{align*}
 	Hence, for $ \norm{U}= 1+\norm{W^{n}}, $ we have $ (f(U), U)\geq 0. $ Therefore, using Theorem \eqref{th2.1} there exist $ v^{*} $ in $ S_{h} $  and $ \alpha>0 $ such that $ f(v^{*})=0$ and $ \norm{v^{*}}\leq \alpha. $
 \end{proof}
 Next lemma deals with the uniqueness of the fully discrete scheme \eqref{21.1}.
 \begin{lemma}
 	\label{L21.3}
 	The approximate solution of discrete scheme \eqref{21.1} is unique.
 \end{lemma}
 \begin{proof}
 	Let $ W_{1}^{n} $ and $ W_{2}^{n} $ are two solution of \eqref{21.1}. Set $ Z^{n}= W_{1}^{n} - W_{2}^{n},$ where $ Z^{n} $ satisfy
 	\begin{align*}
 		\nonumber ({\delta_{t}^{+}}Z^{n},\chi)+\nu (\nabla Z^{n+\theta},\nabla \chi)&+ w_{d}(\nabla Z^{n+\theta}\cdot \mathbf{1},\chi) +(W_{1}^{n+\theta}(\nabla W_{1}^{n+\theta}\cdot \mathbf{1}),\chi)-(W_{2}^{n+\theta}(\nabla W_{2}^{n+\theta}\cdot \mathbf{1}),\chi)\\&
 		+\int_{\partial\Omega}\left(2(c_{2}+w_{d})Z^{n+\theta}+\frac{2}{9c_{2}}\left((W_{1}^{n+\theta})^{3}-(W_{2}^{n+\theta})^{3}\right) \right)\chi d\Gamma=0, \ \ \forall \  \chi \in S_{h}.
 	\end{align*}
 	Selecting $ \chi = Z^{n+\theta} $ to the above equation, we get
 	\begin{align}
 		\label{21.5}
 		\nonumber	\frac{1}{2}{\delta_{t}^{+}}\norm{Z^{n}}^{2}&+ k(\theta-\frac{1}{2})\norm{ {\delta_{t}^{+}}Z^{n}}^{2}+\nu \norm{\nabla Z^{n+\theta}}^{2}+\int_{\partial\Omega}2(c_{2}+w_{d})(Z^{n+\theta})^{2} d\Gamma \nonumber \\&+\frac{2}{9c_{2}}\int_{\partial\Omega}\left((W_{1}^{n+\theta})^{3}-(W_{2}^{n+\theta})^{3}\right)Z^{n+\theta} d\Gamma \nonumber \\&=-w_{d}(\nabla Z^{n+\theta}\cdot \mathbf{1}, Z^{n+\theta})-\left((W_{1}^{n+\theta}(\nabla W_{1}^{n+\theta}\cdot \mathbf{1}), Z^{n+\theta})-(W_{2}^{n+\theta}(\nabla W_{2}^{n+\theta}\cdot \mathbf{1}), Z^{n+\theta})\right).
 	\end{align}
 	The last term on the left hand side of \eqref{21.5} can be written as
 	\begin{align*}
 		\left((W_{1}^{n})^{3}-(W_{2}^{n})^{3}\right)Z^{n}=\frac{(Z^{n})^{2}}{2}\Big( (W_{1}^{n}+ W_{2}^{n})^{2} + (W_{1}^{n})^{2}+ (W_{2}^{n})^{2}\Big).
 	\end{align*}
 	The last term on the right hand side of \eqref{21.5} is bounded by
 	\begin{align*}
 		\Big((W_{1}^{n+\theta}\big(\nabla W_{1}^{n+\theta}\cdot \mathbf{1}), Z^{n+\theta}\big)&-\big(W_{2}^{n+\theta}(\nabla W_{2}^{n+\theta}\cdot \mathbf{1}), Z^{n+\theta}\big)\Big)\\&\leq \norm{Z^{n+\theta}}\norm{\nabla W_{1}^{n+\theta}}\norm{\nabla Z^{n+\theta}}+\norm{W_{2}^{n+\theta}}_{\infty}\norm{\nabla Z^{n+\theta}}\norm{Z^{n+\theta}}.
 	\end{align*}
 	Substituting these values in \eqref{21.5} and using Young's inequality, we arrive at
 	\begin{align*}
 		\frac{1}{2}{\delta_{t}^{+}}\norm{Z^{n}}^{2}&+ k(\theta-\frac{1}{2})\norm{ {\delta_{t}^{+}}Z^{n}}^{2}+\frac{\nu}{2} \norm{\nabla Z^{n+\theta}}^{2}+\int_{\partial\Omega}(c_{2}+w_{d})(Z^{n+\theta})^{2} d\Gamma\leq C\Big(\norm{W_{1}^{n+\theta}}_{1}^{2}+\norm{W_{2}^{n+\theta}}_{1}^{2}\Big)\norm{Z^{n+\theta}}^{2}.
 	\end{align*}
 	For the case $ \theta\geq \frac{1}{2},$ we deduce that
 	\begin{align*}
 		\norm{Z^{n+1}}^{2}+ \nu k \norm{\nabla Z^{n+\theta}}^{2} +k\int_{\partial\Omega}(c_{2}+w_{d})(Z^{n+\theta})^{2} d\Gamma&\leq kC\Big(\norm{W_{1}^{n+\theta}}_{1}^{2}+\norm{W_{2}^{n+\theta}}_{1}^{2}\Big)\theta^{2}\norm{Z^{n+1}}^{2}
 		\\&\quad + \Big(1+kC\big(\norm{W_{1}^{n+\theta}}_{1}^{2}+\norm{W_{2}^{n+\theta}}_{1}^{2}\big)(1-\theta)^{2}\Big)\norm{Z^{n}}^{2}.
 	\end{align*}
 	Suppose that $ k $ is sufficiently small such that $ (1-kC\theta^{2})>0. $ Summing from $ n=0 $ to $M-1 $ and using discrete Gronwall's inequality with Lemma \ref{L21.1}, we get
 	\begin{align*}
 		\norm{Z^{M}}^{2}+ \nu k \sum_{n=0}^{M-1}\norm{\nabla Z^{n+\theta}}^{2} + k\sum_{n=0}^{M-1}\int_{\partial\Omega}(c_{2}+w_{d})(Z^{n+\theta})^{2} d\Gamma\leq 0,
 	\end{align*}
 	since $ Z^{0}=0. $ It follows that, the discrete scheme \eqref{21.1} has unique solution.
 \end{proof}
 
 \section{Error Analysis.}
 In this section, we discuss the error analysis of a fully discrete scheme \eqref{21.1} for the state variable and control input.
 
Introduce the auxiliary projection of \(\tilde{w}_{h}\in V_{h}\) of \(w\) in the following form
\begin{align}
	\label{aux}
       \Big(\nabla(w-\tilde{w}_{h}), \nabla \chi\Big)+\lambda (w-\tilde{w}_{h}, \chi)=0, \quad \chi\in V_{h},
\end{align} 
where \(\lambda\) is a positive constant. The existence and uniqueness of \(\tilde{w}_{h}\) follows from the Lax-Milgram Lemma for given \(w\). Set \(\eta=w-\tilde{w}_{h}\), then the following estimates hold
\begin{align}
	\label{6.2}
	\norm{\eta(t)}_{i}&\leq Ch^{2-i}\norm{w}_{2}, \quad \norm{\eta_{t}(t)}_{i}\leq Ch^{2-i}\norm{w_{t}}_{2},  \quad i=0,1.
\end{align}
For the proof of the above estimate, see \cite{thomee2007galerkin}.

The proof of the following lemma is given in \cite{MR4139150}.
\begin{lemma}
	\label{L6.1}
	Suppose that the boundary of \(\Omega\) is smooth. Then, there holds
	\begin{align*}
		\norm{\eta}_{L^{2}(\partial \Omega)}\leq C h^{\frac{3}{2}}\norm{w}_{2}, \quad \norm{\eta}_{H^{-\frac{1}{2}}(\partial \Omega)}\leq C h^{2}\norm{w}_{2}, \quad \norm{\eta}_{L^{p}(\partial \Omega)}\leq C h\norm{w}_{2}, 
	\end{align*}
where \(2\leq p<\infty.\)
\end{lemma}
 We define the error $ e^{n}:= w(t_{n})-W^{n}=( w(t_{n})-\tilde{w}_{h}(t_{n}))-(W^{n}-\tilde{w}_{h}(t_{n}))=: \eta^{n}-\xi^{n}$ where $ \eta^{n}= w(t_{n})-\tilde{w}_{h}(t_{n}), \ \xi^{n}=W^{n}-\tilde{w}_{h}(t_{n}),$ where \(W^{n}\) is a solution of the fully discrete scheme \eqref{21.1} and $ w(t_{n})(=w^{n})$ is a  solution of the problem \eqref{weak} at \(t=t_{n}\).
 
 Substituting $ t=t_{n+\theta} $ in \eqref{weak} and subtracting  from \eqref{21.1}, we obtain from \eqref{aux} at \( t=t_{n+\theta}\)
 \begin{align}
 	\label{6.3}
 	({\delta_{t}^{+}}\xi^{n},\chi)&+\nu (\nabla \xi^{n+\theta},\nabla \chi)+ \int_{\partial\Omega}2(c_{2} + w_{d})\xi^{n+\theta}\chi d\Gamma + \int_{\partial\Omega}\frac{2}{9c_{2}}(\xi^{n+\theta})^{3}\chi d\Gamma+\int_{\partial\Omega}\frac{2}{3c_{2}}(W^{n+\theta})^{2}\xi^{n+\theta}\chi d\Gamma \nonumber\\&=({\delta_{t}^{+}}\eta^{n}-\lambda\nu \eta^{n+\theta}, \chi)+ w_{d}\left(\nabla (\xi^{n+\theta}-\eta ^{n+\theta})\cdot \mathbf{1},\chi\right) +\int_{\partial\Omega}2(c_{2} + w_{d})\eta^{n+\theta}\chi d\Gamma\nonumber\\&\quad + \left((\eta^{n+\theta}-\xi^{n+\theta})(\nabla w^{n+\theta}\cdot \mathbf{1})+W^{n+\theta}\big(\nabla(\eta^{n+\theta}- \xi^{n+\theta})\cdot \mathbf{1}\big), \chi \right)
 	\nonumber\\&\quad +\frac{2}{9c_{2}}\int_{\partial\Omega}\left((\eta^{n+\theta})^{3}+ 3 w^{n+\theta}\eta^{n+\theta}(w^{n+\theta}-\eta^{n+\theta})+ 3 W^{n+\theta}(\xi^{n+\theta})^{2}\right)\chi d\Gamma+(w_{t}^{n+\theta}-{\delta_{t}^{+}}w^{n}, \chi).
 \end{align}
\begin{lemma}
	\label{L6.2}
	Let \(w_{0}\in H^{3}(\Omega)\) and $ 0\leq \alpha \leq \frac{\beta_{4}\theta^{2}}{8C_{F}} $, where $\theta\in[\frac{1}{2},1]$. Suppose that \(k_{0}>0\) such that 
	\begin{align}
		\label{6.41}
		e^{2\alpha k}\leq 1+k\frac{\beta_{4}\theta^{2}}{4C_{F}}
	\end{align}
 is satisfied for \(0<k\leq k_{0}\), where \(\beta_{4}=\min\{\nu, 2(c_{2}+2w_{d})\}>0\). Then, there exists a positive constant \(C\) independent of \(h\) and \(k\) such that
	\begin{align*}
		\norm{{{\xi}}^{M}}^{2}&+ ke^{-2\alpha t_{M}}\Big(\frac{\beta_{4}}{4}e^{-2\alpha k}-C_{F}\Big(\frac{1-e^{-2\alpha k}}{k}\Big) \Big)\sum_{n=0}^{M-1}\Big(\norm{\nabla \hat{\xi}^{n+1}}^{2}+\norm{\hat{\xi}^{n+1}}_{L^{2}(\partial\Omega)}^{2}\Big)
		\nonumber\\&+\frac{k\theta^{4}}{54c_{2}}e^{-2\alpha t_{M}}\sum_{n=0}^{M-1}e^{2\alpha t_{n}}\norm{\xi^{n+1}}_{L^{4}(\partial\Omega)}^{4}\nonumber\\&\leq e^{-2\alpha t_{M}}C(\norm{w_{0}}_{3})(h^{4}+ k^{2}(\theta-\frac{1}{2})+k^{4}).
	\end{align*}
\end{lemma}
\begin{proof}
	Select \(\chi=\xi^{n+\theta}\) in \eqref{6.3} to have 
	\begin{align}
		\label{6.4}
		({\delta_{t}^{+}}\xi^{n},\xi^{n+\theta})&+\nu\norm{\nabla \xi^{n+\theta}}^{2}+ \int_{\partial\Omega}2(c_{2} + w_{d})(\xi^{n+\theta})^{2} d\Gamma + \frac{2}{9c_{2}}\int_{\partial\Omega}(\xi^{n+\theta})^{4} d\Gamma\nonumber\\&+\frac{2}{3c_{2}}\int_{\partial\Omega}(W^{n+\theta})^{2}(\xi^{n+\theta})^{2} d\Gamma \nonumber\\&=({\delta_{t}^{+}}\eta^{n}-\lambda\nu \eta^{n+\theta}, \xi^{n+\theta})+ w_{d}\Big(\nabla (\eta^{n+\theta}-\xi ^{n+\theta})\cdot \mathbf{1},\xi^{n+\theta}\Big) \nonumber\\&\quad+\int_{\partial\Omega}2(c_{2} + w_{d})\eta^{n+\theta}\xi^{n+\theta} d\Gamma\nonumber\\&\quad + \Big((\eta^{n+\theta}-\xi^{n+\theta})(\nabla w^{n+\theta}\cdot \mathbf{1}), \xi^{n+\theta}\Big) +\Big(W^{n+\theta}\big(\nabla( \eta^{n+\theta}- \xi^{n+\theta})\cdot \mathbf{1}\big), \xi^{n+\theta} \Big)
		\nonumber\\&\quad +\frac{2}{9c_{2}}\int_{\partial\Omega}\Big((\eta^{n+\theta})^{3}+ 3 w^{n+\theta}\eta^{n+\theta}(w^{n+\theta}-\eta^{n+\theta})+ 3 W^{n+\theta}(\xi^{n+\theta})^{2}\Big) \xi^{n+\theta} d\Gamma\nonumber\\&\quad+(w_{t}^{n+\theta}-{\delta_{t}^{+}}w^{n}, \xi^{n+\theta}),
		\nonumber\\&=\sum_{i=1}^{7}	I_{i}.
	\end{align}
On the right hand side of \eqref{6.4}, the first term \(I_{1}\) yields
\begin{align*}
	I_{1}=({\delta_{t}^{+}}\eta^{n}-\lambda\nu \eta^{n+\theta}, \xi^{n+\theta})\leq C\Big( \norm{{\delta_{t}^{+}}\eta^{n}}^{2}+ \norm{\eta^{n+\theta}}^{2}\Big)+ \frac{\epsilon}{18}\norm{\xi^{n+\theta}}^{2},
\end{align*}
where \(\epsilon>0\), we choose later.

Using  Young's inequality, we obtain from second term \(I_{2}\) of \eqref{6.4} with \(\norm{\xi^{n+\theta}}_{H^{\frac{1}{2}}(\partial\Omega)}\leq C\norm{\xi^{n+\theta}}_{H^{1}}\)
\begin{align*}
	I_{2}&= w_{d}\left(\nabla (\eta^{n+\theta}-\xi ^{n+\theta})\cdot \mathbf{1},\xi^{n+\theta}\right),
	\\&=-w_{d}(\eta^{n+\theta}, \nabla \xi^{n+\theta}\cdot \mathbf{1})+\frac{w_{d}}{2}\sum_{i=1}^{2}\int_{\partial\Omega}\eta^{n+\theta}\nu_{i}\xi^{n+\theta}d\Gamma-\frac{w_{d}}{2}\sum_{i=1}^{2}\int_{\partial\Omega}(\xi^{n+\theta})^{2}\nu_{i}d\Gamma,
	\\&\leq C\norm{\eta^{n+\theta}}^{2}+\frac{\nu}{16}\norm{\nabla \xi^{n+\theta}}^{2}+ C\norm{\eta^{n+\theta}}_{H^{-\frac{1}{2}}(\partial\Omega)}^{2}+\frac{\epsilon}{18}\norm{\xi^{n+\theta}}^{2}+w_{d}\norm{\xi^{n+\theta}}_{L^{2}(\partial \Omega)}^{2}.
\end{align*}
From the third term \(I_{3}\), we deduce that
\begin{align*}
	I_{3}= \int_{\partial\Omega}2(c_{2} + w_{d})\eta^{n+\theta}\xi^{n+\theta} d\Gamma\leq C\norm{\eta^{n+\theta}}_{H^{-\frac{1}{2}}(\partial\Omega)}^{2}+ \frac{\nu}{16}\norm{\nabla \xi^{n+\theta}}^{2}+ \frac{\epsilon}{18}\norm{\xi^{n+\theta}}^{2}.
\end{align*}
With the help of the H\"older's inequality and  Gagliardo-Nirenberg inequality, the fourth term \(I_{4}\) on the right hand side of \eqref{6.4} is bounded by
\begin{align*}
	I_{4}&=\left((\eta^{n+\theta}-\xi^{n+\theta})\nabla w^{n+\theta}\cdot \mathbf{1}, \xi^{n+\theta}\right),
	\\&\leq \norm{\eta^{n+\theta}}\norm{\nabla w^{n+\theta}}_{L^{4}(\Omega)} \norm{\xi^{n+\theta}}_{L^{4}(\Omega)} + \norm{\xi^{n+\theta}}\norm{\nabla w^{n+\theta}}_{L^{4}(\Omega)} \norm{\xi^{n+\theta}}_{L^{4}(\Omega)},
	\\&\leq \frac{\nu}{16}\norm{\nabla \xi^{n+\theta}}^{2}+ \frac{\epsilon}{18}\norm{\xi^{n+\theta}}^{2}+ C\norm{\xi^{n+\theta}}^{2}\norm{w^{n+\theta}}_{2}^{2} +C \norm{\eta^{n+\theta}}^{2}\Big(1+\norm{w^{n+\theta}}^{2}+\norm{\Delta w^{n+\theta}}^{2}\Big).
\end{align*}
Setting \(W^{n}=\xi^{n}+\tilde{w}_{h}^{n}\) and using \(\norm{\tilde{w}_{h}^{n}\xi^{n+\theta}}_{H^{\frac{1}{2}}(\partial\Omega)}\leq C\norm{w^{n}}_{2}\norm{\xi^{n+\theta}}_{1}\) with the Sobolev inequality \( \norm{\xi^{n+\theta}}_{L^{4}}\leq \norm{\xi^{n+\theta}}_{1}\) , the  first term of \(I_{5}\)  is estimated by
\begin{align*}
	I_{5}&=\left(W^{n+\theta}\nabla \eta^{n+\theta}\cdot \mathbf{1}, \xi^{n+\theta} \right),\\&=(\xi^{n+\theta}\nabla\eta^{n+\theta}\cdot \mathbf{1}, \xi^{n+\theta})-(\tilde{w}^{n+\theta}\nabla\xi^{n+\theta}\cdot \mathbf{1}, \eta^{n+\theta})-(\eta^{n+\theta}\nabla\tilde{w}_{h}^{n+\theta}\cdot \mathbf{1}, \xi^{n+\theta})
	\\&\quad +\sum_{i=1}^{2}\int_{\partial\Omega}\tilde{w}_{h}^{n+\theta}\eta^{n+\theta}\nu_{i}\xi^{n+\theta}d\Gamma,
	\\&\leq \norm{\xi^{n+\theta}}\norm{\nabla\eta^{n+\theta}}_{L^{4}}\norm{\xi^{n+\theta}}_{L^{4}}+\norm{\eta^{n+\theta}}\norm{\nabla\xi^{n+\theta}}\norm{\tilde{w}^{n+\theta}}_{\infty}+\norm{\eta^{n+\theta}}\norm{\nabla\tilde{w}_{h}^{n+\theta}}_{L^{4}}\norm{\xi^{n+\theta}}_{L^{4}}\\[1mm]&\quad+ \norm{\eta^{n+\theta}}_{H^{-\frac{1}{2}}(\partial\Omega)}\norm{\tilde{w}_{h}^{n+\theta}\xi^{n+\theta}}_{H^{\frac{1}{2}}(\partial\Omega)},
	\\[1mm]&\leq \frac{\nu}{16}\norm{\nabla \xi^{n+\theta}}^{2}+ \frac{\epsilon}{18}\norm{\xi^{n+\theta}}^{2}+C\Big( \norm{w^{n+\theta}}_{2}^{2}\norm{\eta^{n+\theta}}^{2}+\norm{\eta^{n+\theta}}_{H^{-\frac{1}{2}}(\partial\Omega)}^{2}\Big)
	\\[1mm]&\quad +C\norm{\xi^{n+\theta}}^{2}\left(\norm{w^{n+\theta}}^{2}+ \norm{\nabla\eta^{n+\theta}}_{L^{4}}^{2}\right).
\end{align*}
Using \(\norm{\tilde{w}_{h}^{n}}_{\infty}\leq C\norm{w^{n}}_{2}\) and Young's inequality, the second term of \(I_{5}\) yields
\begin{align*}
	I_{5}=-\left(W^{n+\theta}\nabla \xi^{n+\theta}\cdot \mathbf{1}, \xi^{n+\theta} \right)&=\Big((\xi^{n+\theta}+ \tilde{w}_{h}^{n+\theta})\nabla \xi^{n+\theta}\cdot \mathbf{1}, \xi^{n+\theta}\Big)\\&\leq \frac{3c_{2}}{2}\norm{\xi^{n+\theta}}_{L^{2}{(\partial \Omega)}}+ \frac{1}{27c_{2}}\norm{\xi^{n+\theta}}_{L^{4}{(\partial \Omega)}}+ \frac{\nu}{16}\norm{\nabla \xi^{n+\theta}}^{2}\\&\quad+C\norm{w^{n+\theta}}_{2}^{2}\norm{\xi^{n+\theta}}^{2}.
\end{align*}
For \(I_{6}\), it follows that with \(\norm{\xi^{n}}_{L^{2}{(\partial \Omega)}}\leq C\norm{\xi^{n}}_{1}\)
\begin{align*}
	\frac{2}{9c_{2}}\int_{\partial\Omega}(\eta^{n+\theta})^{3}\xi^{n+\theta}d\Gamma\leq C\norm{\eta^{n+\theta}}^{6}_{L^{6}{(\partial \Omega)}}+\frac{\nu}{16}\norm{\nabla \xi^{n+\theta}}^{2}+ \frac{\epsilon}{18}\norm{\xi^{n+\theta}}^{2},
\end{align*}
\begin{align*}
	\frac{2}{9c_{2}}\int_{\partial\Omega} 3 (w^{n+\theta})^{2}\eta^{n+\theta}\xi^{n+\theta} d\Gamma&\leq C\norm{\eta^{n+\theta}}_{H^{-\frac{1}{2}}(\partial\Omega)}\norm{(w^{n+\theta})^{2}\xi^{n+\theta}}_{H^{\frac{1}{2}}(\partial\Omega)},\\&\leq C\norm{\eta^{n+\theta}}_{H^{-\frac{1}{2}}(\partial\Omega)}\norm{w^{n+\theta}}^{2}\norm{\xi^{n+\theta}}_{1},
	\\&\leq C\norm{\eta^{n+\theta}}_{H^{-\frac{1}{2}}(\partial\Omega)}^{2}\norm{w^{n+\theta}}^{4}+ \frac{\nu}{16}\norm{\nabla \xi^{n+\theta}}^{2}+ \frac{\epsilon}{18}\norm{\xi^{n+\theta}}^{2},
\end{align*}
\begin{align*}
	-\frac{2}{9c_{2}}\int_{\partial\Omega} 3 w^{n+\theta}(\eta^{n+\theta})^{2}\xi^{n+\theta} d\Gamma&\leq C\norm{w^{n+\theta}}_{L^{4}(\partial \Omega)}^{2}\norm{\eta^{n+\theta}}_{L^{4}(\partial \Omega)}^{4}+  \frac{\nu}{16}\norm{\nabla \xi^{n+\theta}}^{2}+ \frac{\epsilon}{18}\norm{\xi^{n+\theta}}^{2},
\end{align*}
and 
\begin{align*}
	\frac{2}{3c_{2}}\int_{\partial\Omega}  W^{n+\theta}(\xi^{n+\theta})^{3} d\Gamma\leq \frac{2}{3c_{2}}\int_{\partial\Omega}  (W^{n+\theta})^{2}(\xi^{n+\theta})^{2}+ \frac{2}{12c_{2}}\norm{\xi^{n+\theta}}_{L^{4}(\partial\Omega)}^{4}.
\end{align*}
Finally, the last term $I_{7}$ is bounded by 
\begin{align*}
	I_{7}=(w_{t}^{n+\theta}-{\delta_{t}^{+}}w^{n}, \xi^{n+\theta})\leq C\norm{w_{t}^{n+\theta}-{\delta_{t}^{+}}w^{n}}^{2}+ \frac{\epsilon}{18}\norm{\xi^{n+\theta}}^{2}.
\end{align*}
Substituting the estimates of \(I_{j}, j=1, 2, \ldots,7\) into \eqref{6.4}, we arrive at 
\begin{align}
	\label{6.5}
		({\delta_{t}^{+}}\xi^{n},\xi^{n+\theta})&+\frac{\nu}{2}\norm{\nabla \xi^{n+\theta}}^{2}+ \int_{\partial\Omega}(c_{2} +2 w_{d})(\xi^{n+\theta})^{2} d\Gamma + \frac{1}{54c_{2}}\int_{\partial\Omega}(\xi^{n+\theta})^{4} d\Gamma\nonumber\\&\leq \epsilon\norm{\xi^{n+\theta}}^{2}+C\Big( \norm{{\delta_{t}^{+}}\eta^{n}}^{2}+ \norm{\eta^{n+\theta}}_{L^{6}(\partial\Omega)}^{6} \Big)+ C\norm{\eta^{n+\theta}}^{2}\Big(1+\norm{w^{n+\theta}}^{2}+\norm{\Delta w^{n+\theta}}^{2}\Big)\nonumber\\[1mm]&\quad+ C\Big(\norm{\eta^{n+\theta}}_{H^{-\frac{1}{2}}(\partial\Omega)}^{2}+ \norm{\eta^{n+\theta}}_{L^{4}(\partial\Omega)}^{4}\Big)\left( 1+\norm{w^{n+\theta}}_{2}^{2}+\norm{w^{n+\theta}}_{L^{4}(\partial\Omega)}^{4}\right)\nonumber\\[1mm]&\quad+C\norm{\xi^{n+\theta}}^{2}\Big(\norm{w^{n+\theta}}^{2}+\norm{w^{n+\theta}}_{2}^{2}+\norm{\eta^{n+\theta}}_{L^{4}}^{4}\Big)+C\norm{w_{t}^{n+\theta}-{\delta_{t}^{+}}w^{n}}^{2}.
\end{align}
Note that 
\[({\delta_{t}^{+}}{\xi}^{n}, \xi^{n+\theta})=\frac{1}{2}{\delta_{t}^{+}}\norm{{\xi}^{n}}^{2}+k(\theta- \frac{1}{2})\norm{{\delta_{t}^{+}}{\xi}^{n}}^{2}.
\]
From \eqref{6.5} and \eqref{3.71}, we obtain 
\begin{align}
	\label{6.6}
	{\delta_{t}^{+}}\norm{{\xi}^{n}}^{2}&+2k(\theta- \frac{1}{2})\norm{{\delta_{t}^{+}}{\xi}^{n}}^{2}+\beta_{4}\Big(\norm{\nabla \xi^{n+\theta}}^{2}+\norm{\xi^{n+\theta}}_{L^{2}(\partial\Omega)}^{2}\Big)+  \frac{1}{27c_{2}}\norm{\xi^{n+\theta}}_{L^{4}(\partial\Omega)}^{4}\nonumber\\&\leq 2\epsilon\norm{\xi^{n+\theta}}^{2}+C\Big( \norm{{\delta_{t}^{+}}\eta^{n}}^{2}+ \norm{\eta^{n+\theta}}_{L^{6}(\partial\Omega)}^{6} \Big)+ C\norm{\eta^{n+\theta}}^{2}\Big(1+\norm{w^{n+\theta}}^{2}+\norm{\Delta w^{n+\theta}}^{2}\Big)\nonumber\\[1mm]&\quad+ C\Big(\norm{\eta^{n+\theta}}_{H^{-\frac{1}{2}}(\partial\Omega)}^{2}+ \norm{\eta^{n+\theta}}_{L^{4}(\partial\Omega)}^{4}\Big)\left( 1+\norm{w^{n+\theta}}_{2}^{2}+\norm{w^{n+\theta}}_{L^{4}(\partial\Omega)}^{4}\right)\nonumber\\[1mm]&\quad+C\norm{\xi^{n+\theta}}^{2}\Big(\norm{w^{n+\theta}}^{2}+\norm{w^{n+\theta}}_{2}^{2}+\norm{\eta^{n+\theta}}_{L^{4}}^{4}\Big)+CT^{n+\theta},
\end{align}
where \(\beta_{4}=\min\{\nu, 2(c_{2}+2w_{d})\}\).

Since \(\theta\geq \frac{1}{2}\), we arrive at from \eqref{6.6} using Friedrichs's inequality with \(\epsilon=\frac{\beta_{4}}{4C_{F}}\)
\begin{align*}
	{\delta_{t}^{+}}\norm{{\xi}^{n}}^{2}&+\frac{\beta_{4}}{2}\Big(\norm{\nabla \xi^{n+\theta}}^{2}+\norm{\xi^{n+\theta}}_{L^{2}(\partial\Omega)}^{2}\Big)+  \frac{1}{27c_{2}}\norm{\xi^{n+\theta}}_{L^{4}(\partial\Omega)}^{4}\nonumber\\&\leq C\Big( \norm{{\delta_{t}^{+}}\eta^{n}}^{2}+ \norm{\eta^{n+\theta}}_{L^{6}(\partial\Omega)}^{6} \Big)+ C\norm{\eta^{n+\theta}}^{2}\Big(1+\norm{w^{n+\theta}}^{2}+\norm{\Delta w^{n+\theta}}^{2}\Big)\nonumber\\[1mm]&\quad+ C\Big(\norm{\eta^{n+\theta}}_{H^{-\frac{1}{2}}(\partial\Omega)}^{2}+ \norm{\eta^{n+\theta}}_{L^{4}(\partial\Omega)}^{4}\Big)\left( 1+\norm{w^{n+\theta}}_{2}^{2}+\norm{w^{n+\theta}}_{L^{4}(\partial\Omega)}^{4}\right)\nonumber\\[1mm]&\quad+C\norm{\xi^{n+\theta}}^{2}\Big(\norm{w^{n+\theta}}^{2}+\norm{w^{n+\theta}}_{2}^{2}+\norm{\eta^{n+\theta}}_{L^{4}}^{4}\Big)+CT^{n+\theta}.
\end{align*}
Multiplying by \(e^{2\alpha t_{n+1}}\) and following the proof of Lemma \ref{L21.1}, we have
\begin{align*}
	e^{2\alpha k}{\delta_{t}^{+}}\norm{{\hat{\xi}}^{n}}^{2}&-\Big(\frac{e^{2\alpha k}-1}{k}\Big)\norm{{\hat{\xi}}^{n+1}}^{2}+\frac{\beta_{4}}{4}\theta^{2}\Big(\norm{\nabla \hat{\xi}^{n+1}}^{2}+\norm{\hat{\xi}^{n+1}}_{L^{2}(\partial\Omega)}^{2}\Big)+e^{2\alpha t_{n+1}}\frac{\theta^{4}}{54c_{2}}\norm{\xi^{n+1}}_{L^{4}(\partial\Omega)}^{4}
	\nonumber\\&\leq e^{2\alpha t_{n+1}}(1-\theta)^{2}\Big(\frac{\beta_{4}}{2}\Big(\norm{\nabla {\xi}^{n}}^{2}+\norm{{\xi}^{n}}_{L^{2}(\partial\Omega)}^{2}\Big)+C(1-\theta)^{2}\norm{\xi^{n}}_{L^{4}(\partial\Omega)}^{4}\Big)
	\nonumber\\&\quad+Ce^{2\alpha t_{n+1}}\Big( \norm{{\delta_{t}^{+}}\eta^{n}}^{2}+ \norm{\eta^{n+\theta}}_{L^{6}(\partial\Omega)}^{6} \Big)+ C\norm{\eta^{n+\theta}}^{2}\Big(1+\norm{w^{n+\theta}}^{2}+\norm{\Delta w^{n+\theta}}^{2}\Big)\nonumber\\[1mm]&\quad+ Ce^{2\alpha t_{n+1}}\Big(\norm{\eta^{n+\theta}}_{H^{-\frac{1}{2}}(\partial\Omega)}^{2}+ \norm{\eta^{n+\theta}}_{L^{4}(\partial\Omega)}^{4}\Big)\left( 1+\norm{w^{n+\theta}}_{2}^{2}+\norm{w^{n+\theta}}_{L^{4}(\partial\Omega)}^{4}\right)\nonumber\\[1mm]&\quad+Ce^{2\alpha t_{n+1}}\norm{\xi^{n+\theta}}^{2}\Big(\norm{w^{n+\theta}}^{2}+\norm{w^{n+\theta}}_{2}^{2}+\norm{\eta^{n+\theta}}_{L^{4}}^{4}\Big)+Ce^{2\alpha t_{n+1}}T^{n+\theta}.
\end{align*}
Again, applying Friedrichs's inequality and multiplying by \(e^{-2\alpha k}\) in the resulting inequality, gives
\begin{align*}
	{\delta_{t}^{+}}\norm{{\hat{\xi}}^{n}}^{2}&+ \Big(\frac{\beta_{4}\theta^{2}}{4}e^{-2\alpha k}-C_{F}\Big(\frac{1-e^{-2\alpha k}}{k}\Big) \Big)\Big(\norm{\nabla \hat{\xi}^{n+1}}^{2}+\norm{\hat{\xi}^{n+1}}_{L^{2}(\partial\Omega)}^{2}\Big)
	+e^{2\alpha t_{n}}\frac{\theta^{4}}{54c_{2}}\norm{\xi^{n+1}}_{L^{4}(\partial\Omega)}^{4}
	\nonumber\\&\leq e^{2\alpha t_{n}}(1-\theta)^{2}\Big(\frac{\beta_{4}}{2}\Big(\norm{\nabla {\xi}^{n}}^{2}+\norm{{\xi}^{n}}_{L^{2}(\partial\Omega)}^{2}\Big)+C(1-\theta)^{2}\norm{\xi^{n}}_{L^{4}(\partial\Omega)}^{4}\Big)
	\nonumber\\&\quad+Ce^{2\alpha t_{n}}\Big( \norm{{\delta_{t}^{+}}\eta^{n}}^{2}+ \norm{\eta^{n+\theta}}_{L^{6}(\partial\Omega)}^{6} \Big)+ C\norm{\eta^{n+\theta}}^{2}\Big(1+\norm{w^{n+\theta}}^{2}+\norm{\Delta w^{n+\theta}}^{2}\Big)\nonumber\\[1mm]&\quad+ Ce^{2\alpha t_{n}}\Big(\norm{\eta^{n+\theta}}_{H^{-\frac{1}{2}}(\partial\Omega)}^{2}+ \norm{\eta^{n+\theta}}_{L^{4}(\partial\Omega)}^{4}\Big)\left( 1+\norm{w^{n+\theta}}_{2}^{2}+\norm{w^{n+\theta}}_{L^{4}(\partial\Omega)}^{4}\right)\nonumber\\[1mm]&\quad+Ce^{2\alpha t_{n}}\norm{\xi^{n+\theta}}^{2}\Big(\norm{w^{n+\theta}}^{2}+\norm{w^{n+\theta}}_{2}^{2}+\norm{\eta^{n+\theta}}_{L^{4}}^{4}\Big)+Ce^{2\alpha t_{n}}T^{n+\theta}.
\end{align*}
Choose \(k_{0}>0\) such that \eqref{6.41} is satisfied for \(0<k\leq k_{0}\).

Multiplying by \(k\) and summing from \(n=0\) to \(M-1\) to the above inequality, we observe that
\begin{align*}
	\norm{{\hat{\xi}}^{M}}^{2}&+ k\Big(\frac{\beta_{4}\theta^{2}}{4}e^{-2\alpha k}-C_{F}\Big(\frac{1-e^{-2\alpha k}}{k}\Big) \Big)\sum_{n=0}^{M-1}\Big(\norm{\nabla \hat{\xi}^{n+1}}^{2}+\norm{\hat{\xi}^{n+1}}_{L^{2}(\partial\Omega)}^{2}\Big)
	+\frac{k\theta^{4}}{54c_{2}}\sum_{n=0}^{M-1}e^{2\alpha t_{n}}\norm{\xi^{n+1}}_{L^{4}(\partial\Omega)}^{4}
	\nonumber\\&\leq Ck\sum_{n=0}^{M-1}e^{2\alpha t_{n}}\Big( \norm{{\delta_{t}^{+}}\eta^{n}}^{2}+ \norm{\eta^{n+\theta}}_{L^{6}(\partial\Omega)}^{6} \Big)+ Ck\sum_{n=0}^{M-1}e^{2\alpha t_{n}}\norm{\eta^{n+\theta}}^{2}\Big(1+\norm{w^{n+\theta}}^{2}+\norm{\Delta w^{n+\theta}}^{2}\Big)\nonumber\\[1mm]&\quad+ Ck\sum_{n=0}^{M-1}e^{2\alpha t_{n}}\Big(\norm{\eta^{n+\theta}}_{H^{-\frac{1}{2}}(\partial\Omega)}^{2}+ \norm{\eta^{n+\theta}}_{L^{4}(\partial\Omega)}^{4}\Big)\left( 1+\norm{w^{n+\theta}}_{2}^{2}+\norm{w^{n+\theta}}_{L^{4}(\partial\Omega)}^{4}\right)\nonumber\\[1mm]&\quad+Ck\sum_{n=0}^{M-1}e^{2\alpha t_{n}}\norm{\xi^{n+\theta}}^{2}\Big(\norm{w^{n+\theta}}^{2}+\norm{w^{n+\theta}}_{2}^{2}+\norm{\eta^{n+\theta}}_{L^{4}}^{4}\Big)+Ck\sum_{n=0}^{M-1}e^{2\alpha t_{n}}T^{n+\theta},
\end{align*}
since \(\xi^{0}=0.\) 

Also
 \[\norm{{\delta_{t}^{+}}\eta^{n}}^{2}\leq \frac{1}{k}\int_{t_{n}}^{t_{n+1}}\norm{\eta_{t}}^{2}ds.
\]

We choose \(k\) sufficiently small such that \(1-Ck\theta^{2}>0\). Then using discrete Gronwall's  inequality with \eqref{6.2} and using Lemma \ref{L4.1} with $\alpha=0$ and Lemma \ref{L6.1}, the proof is completed after multiplying by \(e^{-2\alpha t_{M}}\).

\end{proof}
In the following theorem, we discuss the error estimate for the state variable and control input for $\theta\in[\frac{1}{2},1]$.
\begin{theorem}
	Let \(w_{0}\in H^{3}(\Omega)\) and suppose  the hypothesis of Lemma \ref{L6.2} is satisfied. Then there exists a positive constant \(C\) independent of \(h\) and \(k\) such that
	\begin{align*}
		\norm{w^{n}-W^{n}}\leq C(\norm{w_{0}}_{3}) e^{-\alpha t_{M}}\Big(h^{2}+k(\theta-\frac{1}{2})+k^{2}\Big),
	\end{align*} and 
	\begin{align*}
		\norm{v_{2}(t_{n})-V_{2}^{n}}_{L^{2}(\partial \Omega)}\leq C(\norm{w_{0}}_{3}) e^{-\alpha t_{M}}\Big(h^{\frac{3}{2}}+k(\theta-\frac{1}{2})+k^{2}\Big),
	\end{align*}
	where \( V_{2}^{n}=-\frac{1}{\nu}\left(2(c_{2}+w_{d})W^{n}+\frac{2}{9c_{2}}(W^{n})^{3}\right)\).
\end{theorem}
\begin{proof}
The proof is similar to Theorem \ref{th3.1}, applying Lemmas \ref{L6.1} and \ref{L6.2}.
\end{proof}

\section{Numerical Simulation.}
 Next, we discuss some examples to verify our theoretical results in both \(1D\) and \(2D\) cases. We study the behavior of the state variable and control inputs for various values of \(\theta\in [\frac{1}{2}, 1]\). Moreover, we demonstrate that the approximate solution converges to the equilibrium solution over time and investigate the order of convergence for both the state variables and control input.
To solve the nonlinear system of equations for an unknown solution at the current time, we employ Newton's method, using the previous solution as the initial guess. Denote $ w^{n} $ as the refined mesh solution for the state variable and $ v_{0}^{n} $ and $ v_{1}^{n} $, for the control inputs.

In the following example, we discuss the behavior of the state variable and control inputs for the \(1D\) viscous Burgers' equation. The behavior of the uncontrolled (zero Neumann boundary) and controlled solution is studied for various values of  \(c_{0}\) and \(c_{1}\). Moreover, we examine the order of convergence for the state variable and control input corresponding to space and time.
\begin{example}
	\label{ex1}
	In this example, we consider the initial condition \((t=0)\) \(w_{0}(x)= \sin(\pi x)-w_{d}, \ x\in[0,1],\) where \(w_{d}=1.\) Set the diffusion parameter \(\nu=0.1,\) and \(c_{0}=0.1, \ c_{1}=0.1\). 
\end{example}
We consider the \(1D\) viscous Burgers' equation \eqref{Eq1}-\eqref{Eq2} with zero Neumann boundary (Uncontrolled solution) and the system \eqref{w11}-\eqref{w12} with nonlinear Neumann boundary feedback control laws. We solve the fully discrete scheme \eqref{2.2} with \(\theta=1.\) Figure \ref{fig:ex1}(i) shows the uncontrolled solution denoted as ``Uncontrolled solution" and controlled solution from which we observe that large values of  \(c_{0}\) and \(c_{1}\) lead to a faster decay in the \(L^{2}\)-norm. Therefore, the simulation confirms that the solution of \eqref{Eq1}-\eqref{Eq2} with control laws converges to the equilibrium solution for various values of \(c_{0}\) and \(c_{1}\).

 The control inputs \(v_{0}(t)\) and \(v_{1}(t)\) are demonstrated in Figures \ref{fig:ex1}(iv) and (v), where we notice that for large values of \(c_{0}\) and \(c_{1}\), the control inputs converge faster to zero with respect to time. From Figure \ref{fig:ex1}(ii), we observe that the state variable in the \(L^{2}\)-norm decays rapidly for \(\nu=1\). For \(\theta\in[\frac{1}{2}, 1]\), the state variable is depicted in Figure \ref{fig:ex1}(iii).
\begin{figure}[!ht]
	\centering
	(i) \includegraphics[width= 0.30\textwidth]{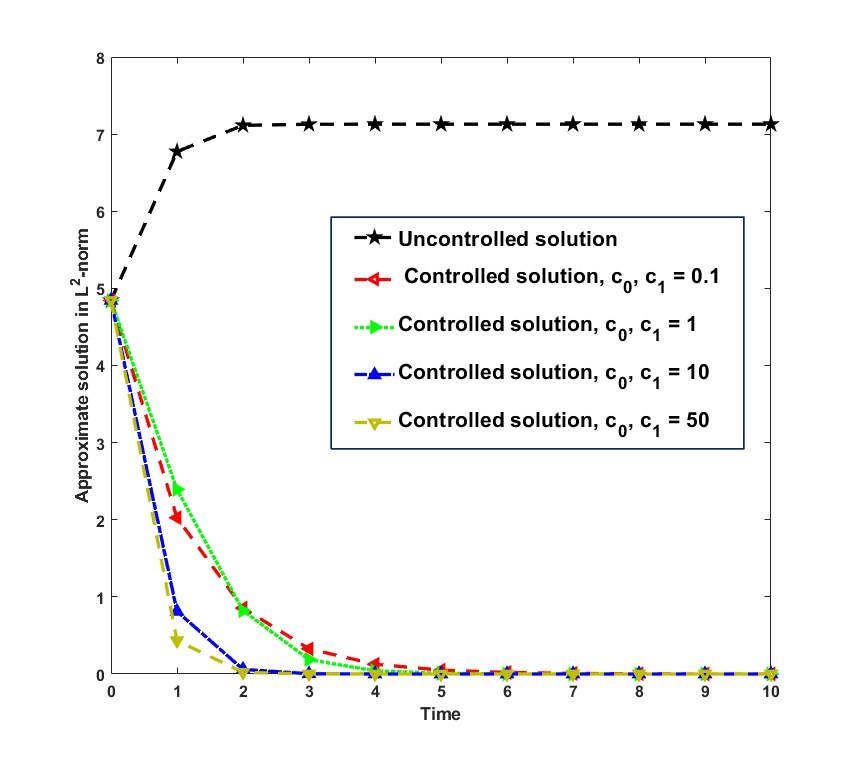}
	(ii) \includegraphics[width= 0.30\textwidth]{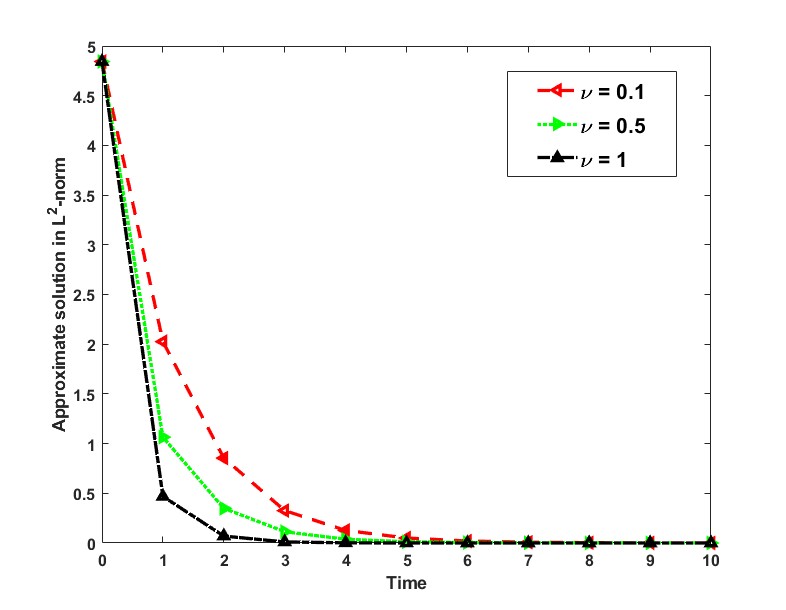}	
	(iii)\includegraphics[width= 0.35\textwidth]{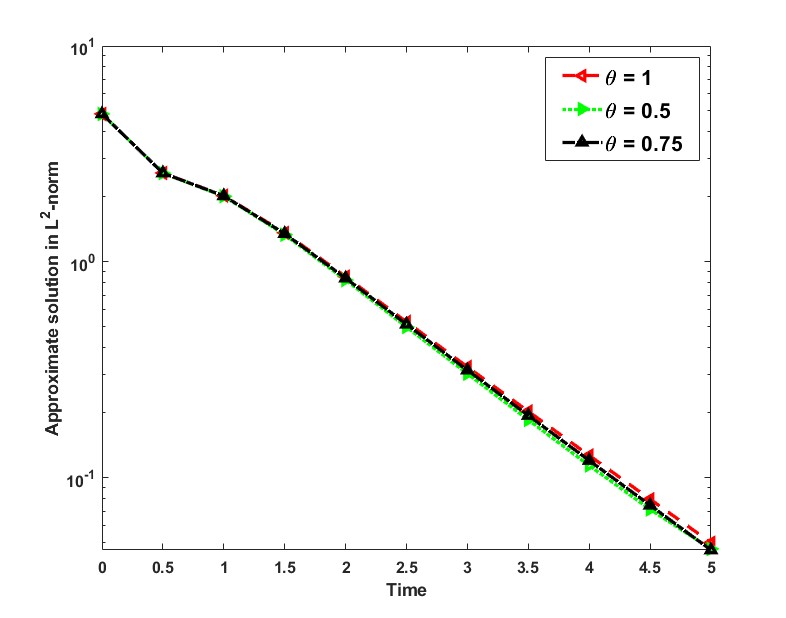}
	(iv)\includegraphics[width= 0.30\textwidth]{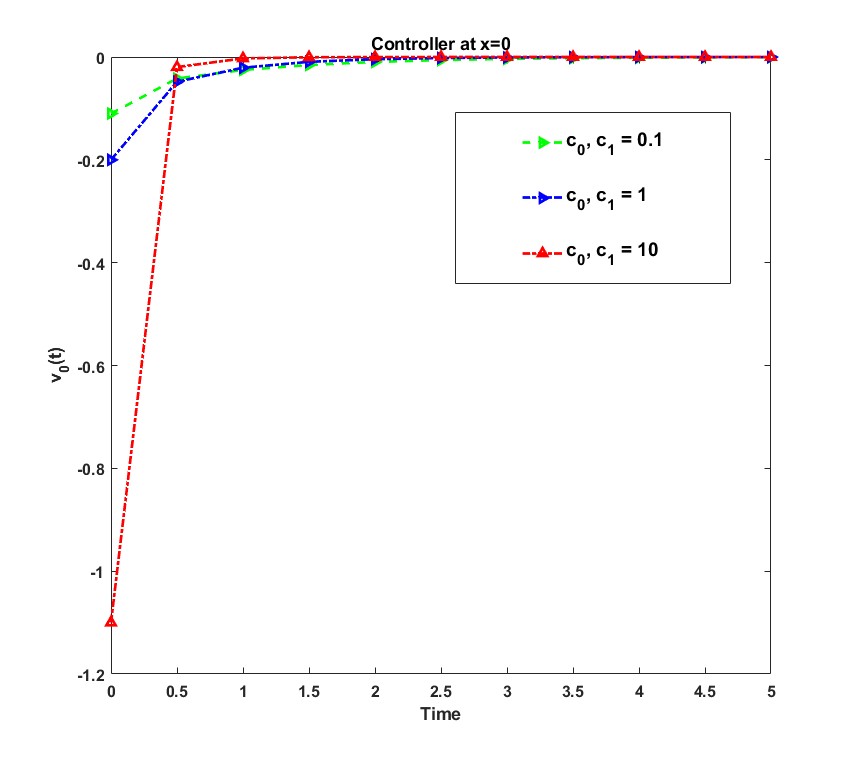}
	(v) \includegraphics[width= 0.30\textwidth]{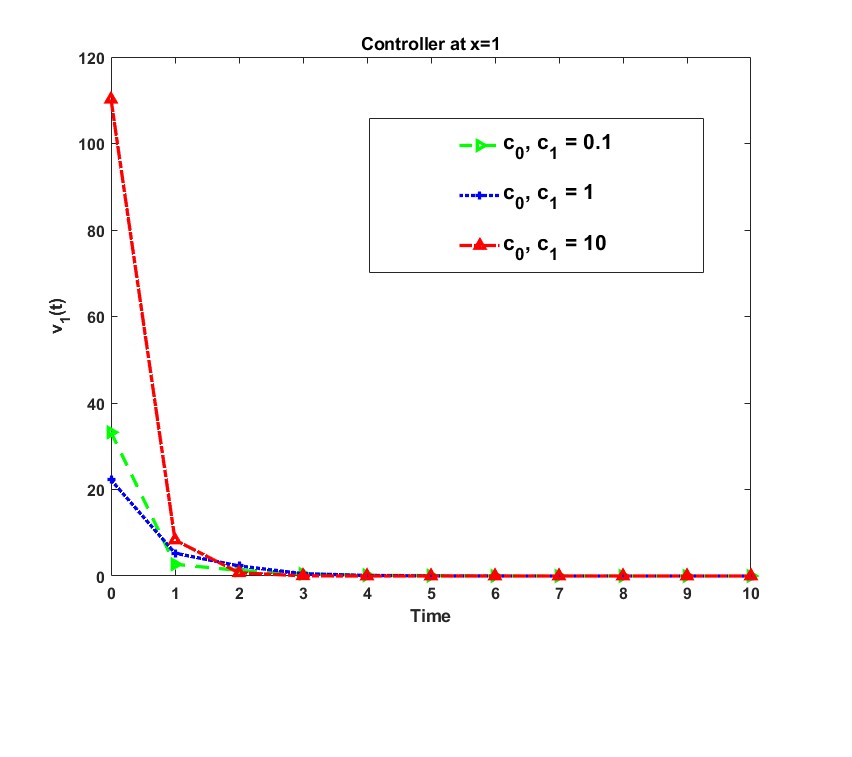}
	\caption{Example \ref{ex1}: {\bf (i)} Controlled and uncontrolled solution in the \(L^{2}\)-norm.  
		{\bf (ii)} Controlled solution in the \(L^{2}-\)norm for various values of \(\nu\).  
		{\bf (iii)} In semi-log scale, controlled solution in the \(L^{2}-\)norm for different values of \(\theta\). 
		{\bf (iv)} Control input for numerous values of \( c_{0} \ \text{and} \ c_{1} \) at the left boundary \(x=0\).  
		{\bf (v)} Control input for various values of \( c_{0}\ \text{and} \ c_{1} \)  at the right boundary \(x=1\).  
   }  
	\label{fig:ex1}
\end{figure}

 The control inputs \(v_{0}(t)\) and \(v_{1}(t)\) decay faster for large values of \(\nu\) as shown in Figures \ref{fig:ex2}(i) and (ii).  In Figures \ref{fig:ex2}(iii) and (iv), the control inputs decay towards zero with respect to time for \(\theta\in[\frac{1}{2}, 1]\).
 \begin{figure}[!ht]
 	\centering
 	(i) \includegraphics[width= 0.43\textwidth]{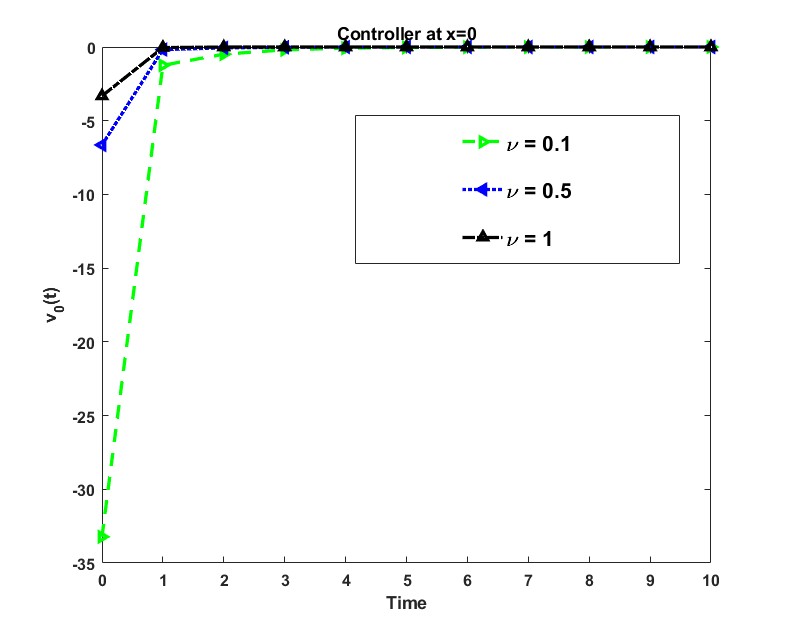}
 	(ii) \includegraphics[width= 0.43\textwidth]{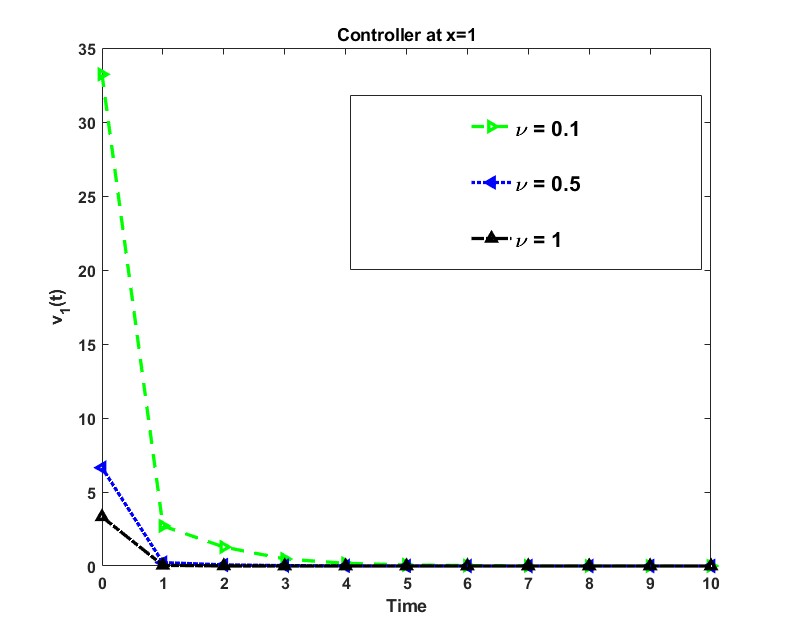}	
 	(iii)\includegraphics[width= 0.43\textwidth]{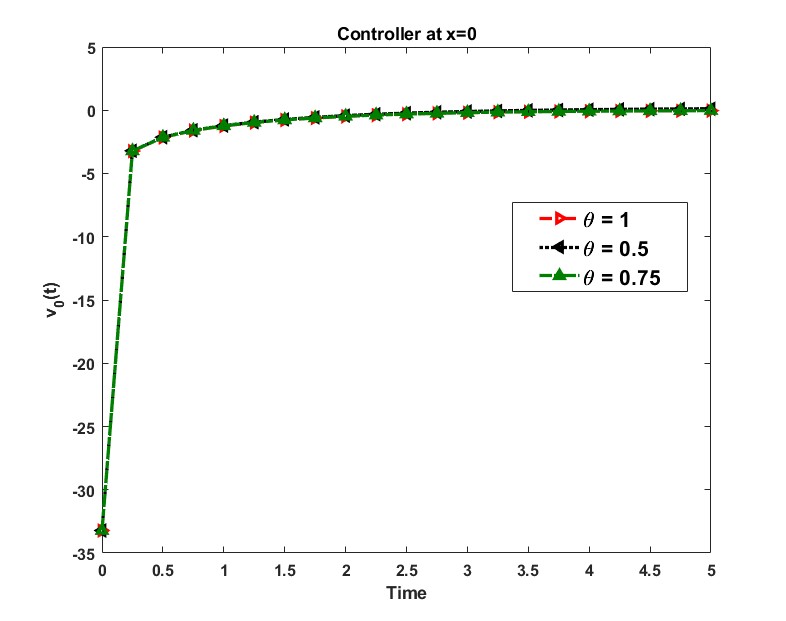}
 	(iv)\includegraphics[width= 0.43\textwidth]{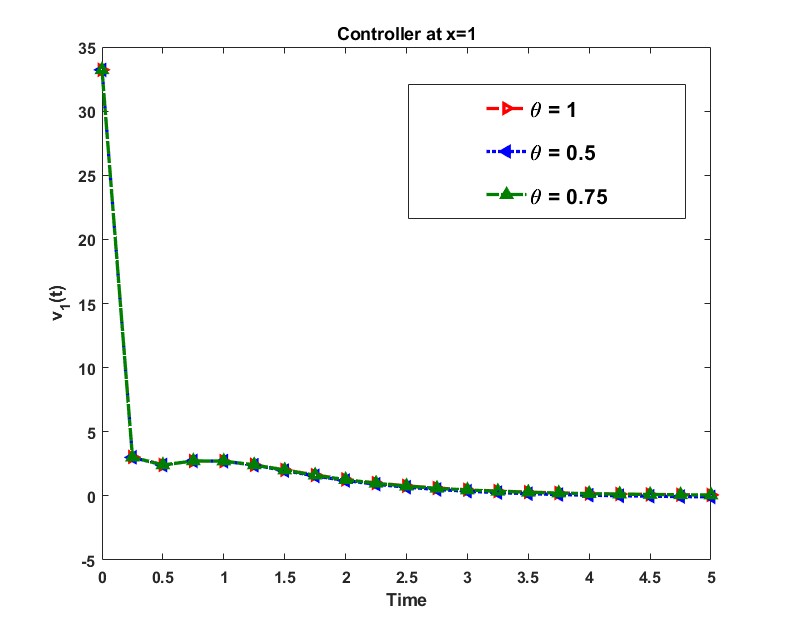}
 	\caption{Example \ref{ex1}: {\bf (i)} Control input for numerous values of \( \nu \) at the left boundary \(x=0\).  
 		{\bf (ii)} Control input for different values of \( \nu \) at the right boundary \(x=1\).  
 		{\bf (iii)} Control input for various values of \( \theta  \) at the left boundary \(x=0\). 
 		{\bf (iv)}  Control input for various values of \( \theta  \) at the right boundary \(x=1\).    
 	}  
 	\label{fig:ex2}
 \end{figure}

Next, we examine the order of convergence for the state variable and control inputs in the time interval $ [0,1] $. From Table \ref{table:1}, we observe that the error for the state variable decreases in both the \(L^{2}\) and \(L^{\infty}\)-norms, and the order of convergence is two for various values of \(h\) with the fixed value of \(k\). For the order of convergence, the refined mesh solution is considered an exact solution because the exact solution is not known for this equation. In Table \ref{table:2}, we present the error and order of convergence in space for both control inputs, and the order of convergence is two in the \(L^{\infty}\)-norm.

 Table \ref{table:3} contains the order of convergence of the state variables about time for different values of \(k\) with a fixed value of \(h.\) We notice that, for \(\theta=1\), the order of convergence is one, while for \(\theta=\frac{1}{2}\), the order of convergence is two in the \(L^{\infty}\)-norm, which concludes our theoretical results of Lemma \ref{L3.1}. The order of convergence for the control inputs \(v_{0}(t)\) and \(v_{1}(t)\) is shown in Tables \ref{table:4} and \ref{table:5}, respectively. A similar type order of convergence is obtained as in Table \ref{table:3}. Moreover, we observe that the order of convergence with respect to time is one for \(\theta\in(\frac{1}{2}, 1)\).
	
\begin{table}[h]
	\centering
	\caption{ The order of convergence (O.C.) of the state variable with respect to space in Example \ref{ex1}, considering varying values of $h$ and a fixed value of $M=100$.}
	\begin{tabular}{ c ||c|| c|| c|| c }
		\toprule
		h & $\norm{W^{n}-w^{n}}$         & O. C. & $\norm{W^{n}-w^{n}}_{\infty}$    &          O. C.    \\
		\midrule
		$\frac{1}{4}$& $ 7.6454e-04 $ &  $-- $ & $8.3202e-04$ &   $--$     \\
	
		$\frac{1}{8}$ &  $1.7592e-04 $ & $2.12$ & $1.9793e-04$&     $2.07$   \\
		
		$\frac{1}{32}$ &  $ 4.2590e-05  $ & $2.04$ & $4.8198e-05$  &  $2.03$   \\
		
		$\frac{1}{16}$ & $1.0484e-06 $  & $ 2.02 $ &  $1.1964e-05$  &  $2.01$  \\
		
		$\frac{1}{32}$ &  $ 2.5779e-06$  & $2.02$ &  $2.9595e-06$  &  $2.01$  \\
		\bottomrule
	\end{tabular}
	\label{table:1}
\end{table}

\begin{table}[ht]
	\centering
	\caption{ The order of convergence (O.C.) of the control inputs with respect to space in Example \ref{ex1}, considering varying values of $h$ and a fixed value of $M=100$.}
	\begin{tabular}{ c ||c|| c|| c|| c }
		\toprule
		h & $\norm{V_{0}^{n}-v_{0}^{n}}_{\infty}$        & O. C. & $\norm{V_{1}^{n}-v_{1}^{n}}_{\infty}$   &          O. C.    \\
		\midrule
		$\frac{1}{8}$& $ 5.2e-03 $ &  $-- $ & $0.0071$ &   $--$     \\
		
		$\frac{1}{16}$ &  $1.2e-03 $ & $2.13$ & $0.0018$&     $2.01$   \\
		
		$\frac{1}{32}$ &  $ 2.937e-04 $ & $2.03$ & $4.3809e-04$  &  $2.00$   \\
		
		$\frac{1}{64}$ & $7.2843e-05 $  & $ 2.01 $ &  $1.0979e-04$  &  $1.99$  \\
		
		$\frac{1}{128}$ &  $ 1.7971e-05$  & $2.01$ &  $2.7447e-05$  &  $2.00$  \\
		\bottomrule
	\end{tabular}
	\label{table:2}
\end{table}

\begin{table}[!h]
	\centering
	\caption{ The order of convergence (O.C.) of the state variable with respect to time in Example \ref{ex1}, considering varying values of $k$ and a fixed value of $h=\frac{1}{30}$.}
	\begin{tabular}{ c ||c|| c|| c|| c||c }
		\toprule
		k & $\norm{W^{n}-w^{n}}_{\infty}$, for   \(\theta=1\)      & O. C. & k & $\norm{W^{n}-w^{n}}_{\infty}$, for   \(\theta=\frac{1}{2}\)   &          O. C.    \\
		\midrule
		$\frac{1}{8}$ &  $ 4.7821e-04$  &  $--  $&$\frac{1}{40}$  &  $1.9935e-06 $ &    $--$    \\
		
		$\frac{1}{16}$& $ 2.6903e-04 $ &  $0.84 $& $\frac{1}{80}$  & $4.3275e-07$ &   $2.20$     \\
		
		$\frac{1}{32}$ &  $1.4173e-04 $ & $0.92$ &$\frac{1}{160}$& $1.0806e-07$&     $2.00$   \\
		
		$\frac{1}{64}$ &  $ 7.2224e-05 $ & $0.97$&$\frac{1}{320}$ & $2.692e-08$  &  $2.00$   \\
		
		$\frac{1}{128}$ & $3.6072e-05 $  & $ 1.00 $&$\frac{1}{640}$ &  $6.6423e-09$  &  $2.01$  \\
		
		$\frac{1}{256}$ &  $ 1.7663e-05$  & $1.03$&$\frac{1}{1280}$ &  $1.5746e-09$  &  $2.07$  \\
		\bottomrule
	\end{tabular}
	\label{table:3}
\end{table}
\begin{table}[!ht]
	\centering
	\caption{ The order of convergence (O.C.) of the control input at the left boundary \((x=0)\) with respect to time in Example \ref{ex1}, considering varying values of $k$ and a fixed value of $h=\frac{1}{30}$.}
	\begin{tabular}{ c ||c|| c|| c|| c||c }
		\toprule
		k & $\norm{V_{0}^{n}-v_{0}^{n}}_{\infty}$, \ \(\theta=1\)      & O. C.& k & $\norm{V_{0}^{n}-v_{0}^{n}}_{\infty}$, \   \(\theta=\frac{1}{2}\)   &          O. C.    \\
		\midrule
	$\frac{1}{8}$ &  $ 0.0016$  &  $--  $&$\frac{1}{40}$  &  $5.5338e-07 $ &    $--$    \\
	
	$\frac{1}{16}$& $ 8.1063e-04 $ &  $0.99 $& $\frac{1}{80}$  & $4.3754e-07$ &   $0.34$     \\
	
	$\frac{1}{32}$ &  $4.0617e-04 $ & $0.99$ &$\frac{1}{160}$& $1.0952e-07$&     $1.99$   \\
	
	$\frac{1}{64}$ &  $ 2.0223e-04 $ & $1.00$&$\frac{1}{320}$ & $2.7315e-08$  &  $2.00$   \\
	
	$\frac{1}{128}$ & $1.00e-04 $  & $ 1.01 $&$\frac{1}{640}$ &  $6.7527e-09$  &  $2.01$  \\
	
	$\frac{1}{256}$ &  $ 4.8729e-05$  & $1.03$&$\frac{1}{1280}$ &  $1.6184e-09$  &  $2.06$  \\
		\bottomrule
	\end{tabular}
	\label{table:4}
\end{table}
\begin{table}[H]
	\centering
	\caption{ The order of convergence (O.C.) of the control input at the right boundary \((x=1)\) with respect to time in Example \ref{ex1}, considering varying values of $k$ and a fixed value of $h=\frac{1}{30}$.}
	\begin{tabular}{ c ||c|| c|| c|| c|| c }
		\toprule
		k & $\norm{V_{1}^{n}-v_{1}^{n}}_{\infty}$, for   \(\theta=1\)      & O. C.& k & $\norm{V_{1}^{n}-v_{1}^{n}}_{\infty}$, for   \(\theta=\frac{1}{2}\)   &          O. C.    \\
		\midrule
	$\frac{1}{8}$ &  $ 0.0052$  &  $--  $&$\frac{1}{40}$  &  $2.1931e-05 $ &    $--$    \\
	
	$\frac{1}{16}$& $ 0.0029 $ &  $0.83 $& $\frac{1}{80}$  & $4.6023e-06$ &   $2.25$     \\
	
	$\frac{1}{32}$ &  $0.0015 $ & $0.92$ &$\frac{1}{160}$& $1.1491e-06$&     $2.00$   \\
	
	$\frac{1}{64}$ &  $ 7.8265e-04 $ & $0.97$&$\frac{1}{320}$ & $2.8627e-07$  &  $2.00$   \\
	
	$\frac{1}{128}$ & $3.9104e-04 $  & $ 1.00 $&$\frac{1}{640}$ &  $7.0633e-08$  &  $2.01$  \\
	
	$\frac{1}{256}$ &  $ 1.9152e-04$  & $1.03$&$\frac{1}{1280}$ &  $1.6745e-08$  &  $2.07$  \\
		\bottomrule
	\end{tabular}
	\label{table:5}
\end{table}
In the next example, we illustrate the behavior of the \(2D\) viscous Burgers' equation to the state variable and control input for different values of \(c_{2}\) and \(\nu\).
\begin{example}
	\label{ex2}
	Here, we select the initial condition \(w_{0}(x_{1},x_{2})=5x_{1}(1-x_{1})x_{2}(1-x_{2})-w_{d}\), \((x_{1},x_{2})\in [0, 1]\times[0,1]\), where \(w_{d}=2\). Set the diffusion coefficient \(\nu=1\) with the time interval \(t=[0, 1]\).
\end{example}
We consider the \(2D\) uncontrolled \((v_{2}(t)=0)\) viscous Burgers' equation \eqref{11.1}-\eqref{11.3} and the system \eqref{11.7}-\eqref{11.9} with nonlinear Neumann boundary feedback control input. We solve \eqref{21.1} with \(\theta=1.\) In Figure \ref{fig:ex3}(i), we notice that the uncontrolled solution (zero Neumann boundary) does not go to zero, denoted as ``Uncontrolled solution". But in the case of feedback control law, the controlled solution goes to zero for various values of \(c_{2}\) in the $ L^{2}$-norm, which shows in the same figure. Hence, the system \eqref{11.7}-\eqref{11.9} with control input $ v_{2}(x,t)$ is stabilizable for \(c_{2}>0\). Moreover, the simulation confirms that the Burgers' equation \eqref{11.1}-\eqref{11.3} goes to the steady state solution of the system \eqref{11.4}-\eqref{11.5} over time.
 In Figure \ref{fig:ex3}(ii), we depict the behavior of the control input for numerous values of \(c_{2}\) in the \(L^{2}\)-norm. Figures \ref{fig:ex3}(iii) and (iv) illustrate the state variable and control input with different values of \(\nu\) with \(c_{2}=0.1\), respectively.
 \begin{figure}[!ht]
	\centering
	(i) \includegraphics[width= 0.44\textwidth]{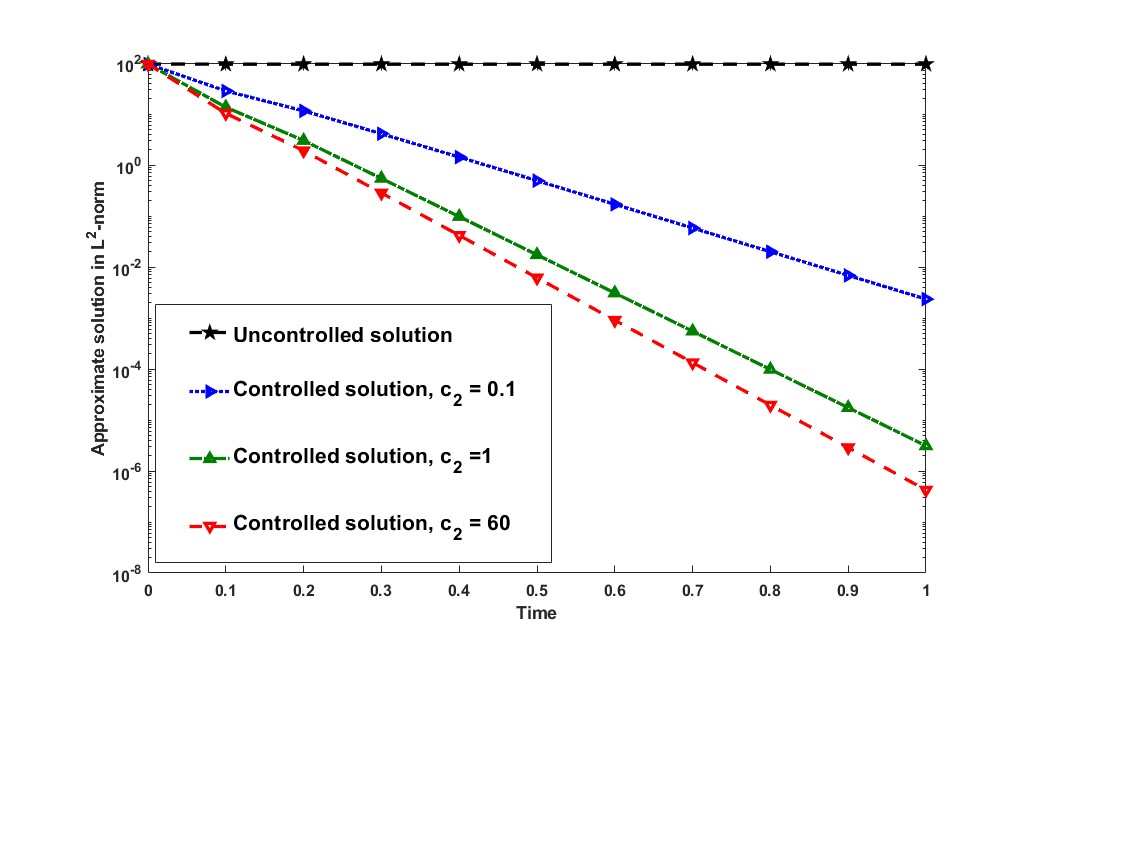}
	(ii) \includegraphics[width= 0.36\textwidth]{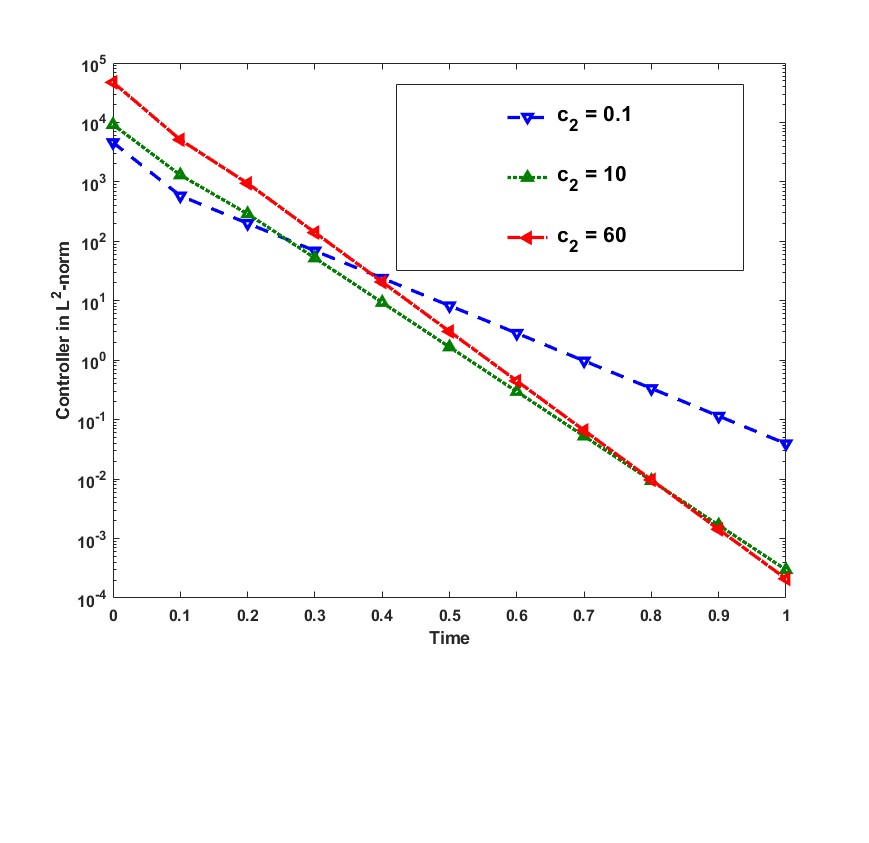}	
	(iii)\includegraphics[width= 0.35\textwidth]{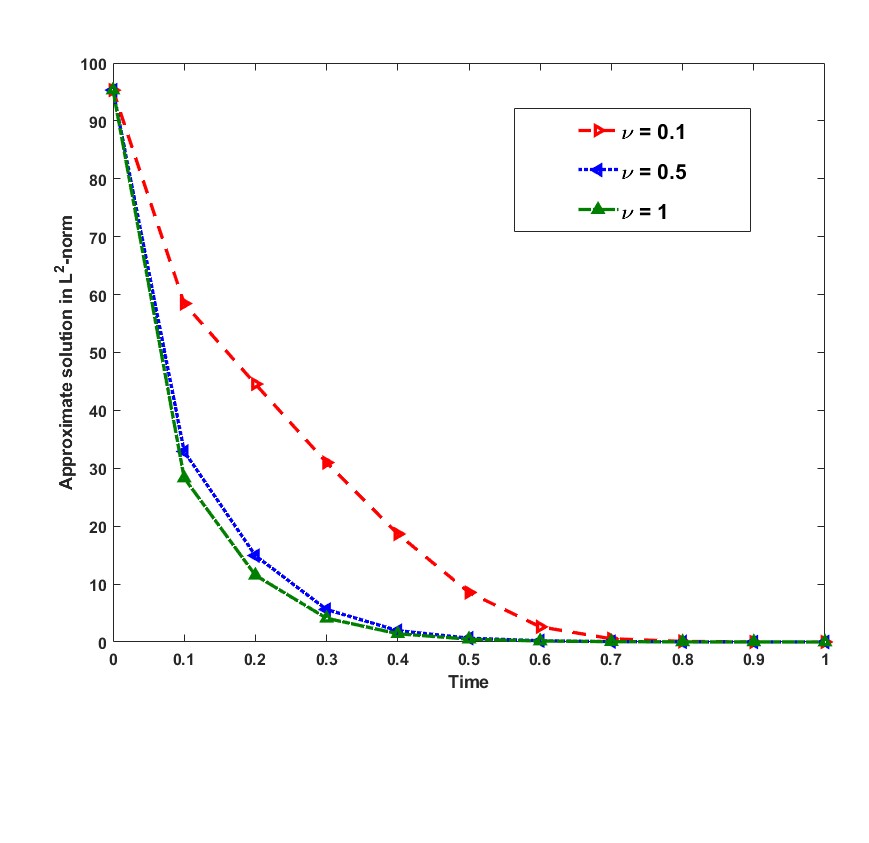}
	(iv)\includegraphics[width= 0.43\textwidth]{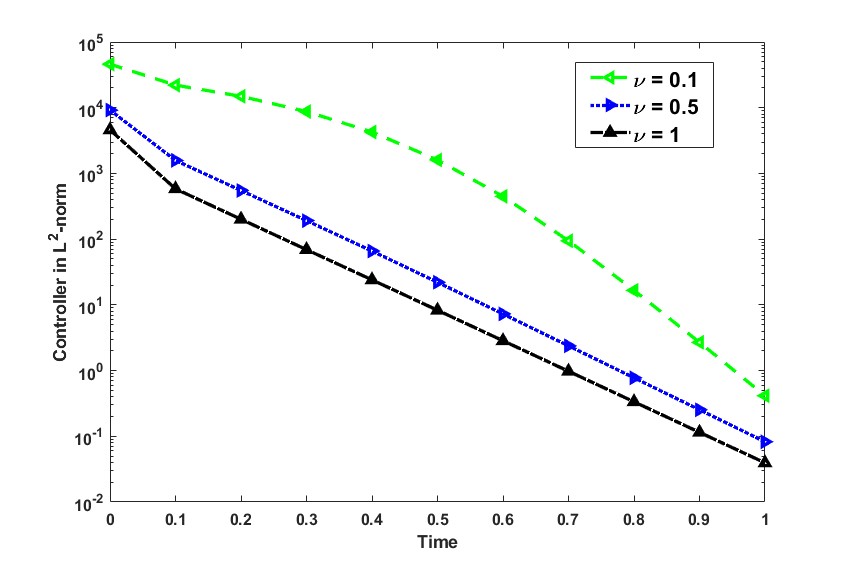}
	\caption{Example \ref{ex2}: {\bf (i)} In semi-log scale, uncontrolled and controlled solution in the \(L^{2}-\)norm.  
		{\bf (ii)} In semi-log scale, control input for different values of \( c_{2} \) with \(\nu=1\) in the \(L^{2}\)-norm.  
		{\bf (iii)} Controlled solution in the \(L^{2}\)-norm for various values of \(\nu\) with \(c_{2}=0.1\). 
		{\bf (iv)} In semi-log scale, control input for various values of \( \nu  \) in the \(L^{2}-\)norm with \(c_{2}=0.1\).    
	}  
	\label{fig:ex3}
\end{figure}

\section*{Acknowledgments}
Sudeep Kundu gratefully acknowledges the support of the Science \& Engineering Research Board (SERB), Government of India, under the Start-up Research Grant, Project No. SRG/2022/000360.
\section*{Declarations}

\textbf{CONFlLICT OF INTEREST.} 

The authors declare no conflict of interest.

\section{Appendix.}
The remaining proof of Theorem \ref{L2.12} is proved in the following parts:

\textbf{Part 1.} In this part, we show $w_{tx}\in L^{\infty}(0,T;L^{2}(0,1))$ and $w_{tt}\in L^{2}(0, T; L^{2}(0,1))$
\begin{align*}
	\nu \norm{w_{tx}}^{2}+ E_{3}(t)+  e^{-2\alpha t} \int_{0}^{t}e^{2\alpha s}\norm{w_{tt}(s)}^{2}ds\leq e^{-2\alpha t}C( \norm{w_{0}}_{3}) \exp(C\norm{w_{0}}_{2}^{2}).
\end{align*}
\begin{proof}
	Differentiating \eqref{1.9} with respect to time and setting \(\phi=w_{tt}\) yields
	\begin{align}
		\label{2.23}
		\frac{1}{2}\frac{d}{dt}\left(\nu \norm{w_{tx}}^{2}+E_{3}(t)\right)+ \norm{w_{tt}}^{2}&=\frac{2}{3c_{0}}w(0,t)w_{t}^{3}(0,t)+\frac{2}{3c_{1}}w(1,t)w_{ht}^{3}(1,t)\nonumber\\& \quad - (w_{t}w_{x}+ww_{tx}, w_{tt})-w_{d}(w_{tx},\phi).
	\end{align}
	Using Young's inequality on the first-two terms on the right hand side of \eqref{2.23} gives
	\begin{align*}
		\frac{2}{3c_{i}}w(i,t)w_{t}^{3}(i,t)\leq \frac{2}{3c_{i}}w^{2}(i,t)w_{t}^{2}(i,t)+\frac{1}{6c_{i}}w_{t}^{4}(i,t),\quad  i=0,1.
	\end{align*}
	With the help of the Cauchy-Schwarz inequality, the last two terms on the right hand side of \eqref{2.23} are bounded by 
	\begin{align*}
		- (w_{t}w_{x}+ww_{tx}, w_{tt})-w_{d}(w_{tx},w_{tt})&\leq \norm{w_{t}}_{\infty}\norm{w_{x}}\norm{w_{tt}}+\norm{w}_{\infty}\norm{w_{tx}}\norm{w_{tt}}+w_{d}\norm{w_{tx}}\norm{w_{tt}},
		\\&\leq Cw_{t}^{2}(0,t)+ C \norm{|w|}^{2}\norm{w_{tx}}^{2}+ C\norm{w_{tx}}^{2}+\frac{1}{2}\norm{w_{tt}}^{2}.
	\end{align*}
	Substituting these estimates into \eqref{2.23} and multiplying by \(2e^{2\alpha t}\) in the resulting inequality, we arrive at
	\begin{align*}
		\frac{d}{dt}\left(\nu e^{2\alpha t}\norm{w_{tx}}^{2}+e^{2\alpha t} E_{3}(t)\right)+ e^{2\alpha t}\norm{w_{tt}}^{2}&\leq Ce^{2\alpha t}w_{t}^{2}(0,t)+ Ce^{2\alpha t} \norm{|w|}^{2}\norm{w_{tx}}^{2}\\&\quad+ \frac{2}{3c_{0}}w^{2}(0,t)w_{t}^{2}(0,t)+\frac{1}{6c_{0}}w_{t}^{4}(0,t)
		\\&\quad+ \frac{2}{3c_{1}}w^{2}(1,t)w_{t}^{2}(1,t)+\frac{1}{6c_{1}}w_{t}^{4}(1,t)
		\\&\quad + Ce^{2\alpha t} \norm{w_{tx}}^{2}+ 2\alpha e^{2\alpha t}\left(\nu \norm{w_{tx}}^{2}+ E_{3}(t)\right).
	\end{align*}
	Integrating  with respect to time over \([0,t]\) and using Theorem \ref{L2.12}, we observe that
	\begin{align*}
		\nu e^{2\alpha t}\norm{w_{tx}}^{2}+e^{2\alpha t} E_{3}(t)+ \int_{0}^{t}e^{2\alpha s}\norm{w_{tt}(s)}^{2}ds&\leq C( \norm{w_{0}}_{3}) \exp(C\norm{w_{0}}_{2}^{2})\\&\quad+\norm{w_{tx}(0)}^{2}.
	\end{align*}
	From \eqref{w11}, we can write after setting \(t=0\)
	\begin{align*}
		\norm{w_{tx}(0)}\leq C\left( \norm{w_{0}}_{3}+ \norm{w_{0}}_{1} \norm{w_{0}}_{2}  \right).
	\end{align*}
	Therefore, the proof is completed after multiplying by \( e^{-2\alpha t}\) in the resulting inequality.
\end{proof}
 \textbf{Part 2.} Here, we show the following exponential bound 	\begin{align*}
 	\norm{w_{tt}}^{2} +e^{-2\alpha t} \int_{0}^{t} e^{2\alpha s} \Big( \nu \norm{w_{ttx}}^{2} + E_{4}(s)\Big)ds\leq e^{-2\alpha t}C(\norm{w_{0}}_{3}) \exp(C\norm{w_{0}}_{2}^{2}),
 \end{align*}
where \[ E_{4}(t)= (c_{0}+2w_{d})w_{tt}^{2}(0,t)+2(c_{1}+w_{d})w_{tt}^{2}(1,t)+\frac{2 }{3c_{0}}w^{2}(0,t)w_{tt}^{2}(0,t)+\frac{2}{3c_{1}}w^{2}(1,t)w_{tt}^{2}(1,t).
\]
\begin{proof}
	Differentiating \eqref{1.9} with respect to time and choosing \(\phi=w_{tt}\) gives
	\begin{align}
		\label{2.25}
		\frac{1}{2}\frac{d}{dt}\norm{w_{tt}}^{2}+\nu \norm{w_{ttx}}^{2}&+ (c_{0}+w_{d})w_{tt}^{2}(0,t)+(c_{1}+w_{d})w_{tt}^{2}(1,t)\nonumber\\&+\frac{2}{3c_{0}}w^{2}(0,t)w_{tt}^{2}(0,t)+\frac{2}{3c_{1}}w^{2}(1,t)w_{tt}^{2}(1,t)\nonumber\\&= -\frac{4}{3c_{0}}w(0,t)w_{t}^{2}(0,t)w_{tt}(0,t)-\frac{4}{3c_{1}}w(1,t)w_{t}^{2}(1,t)w_{tt}(1,t)
		\nonumber\\&\quad -(w_{tt}w_{x}+2w_{ht}w_{tx}+ww_{ttx},w_{tt})-w_{d}(w_{ttx},w_{tt}).
	\end{align}
	The first two terms on the right hand side of \eqref{2.25} is estimated by
	\begin{align*}
		\frac{4}{3c_{i}}w(i,t)w_{t}^{2}(i,t)w_{tt}(i,t)\leq \frac{2}{6c_{i}}w^{2}(i,t)w_{tt}^{2}(i,t)+\frac{4}{3c_{i}}w_{t}^{4}(i,t), \quad i=0,1.
	\end{align*}
	On the right hand side of \eqref{2.25}, the third term  yields
	\begin{align*}
		-(w_{tt}w_{x}+2w_{t}w_{tx}+ww_{ttx},w_{tt})&\leq \norm{w_{tt}}_{\infty}\norm{w_{x}}\norm{w_{tt}}+2\norm{w_{t}}_{\infty}\norm{w_{tx}}\norm{w_{tt}}+\norm{w}_{\infty}\norm{w_{ttx}}\norm{w_{tt}},
		\\&\leq (w_{tt}(0,t)+\norm{\norm{w_{ttx}}})\norm{w_{x}}\norm{w_{tt}}\\&\quad +2(w_{t}(0,t)+\norm{w_{tx}})\norm{w_{tx}}\norm{w_{tt}}+\sqrt{2}\norm{|w|}\norm{w_{ttx}}\norm{w_{tt}}.
	\end{align*}
	With the help of Young's inequality, we write the above inequality as
	\begin{align*}
		-(w_{tt}w_{x}+2w_{t}w_{tx}&+ww_{ttx},w_{tt})\\&\leq \frac{c_{0}}{2}w_{tt}^{2}(0,t)+ \frac{\nu}{4}\norm{w_{ttx}}^{2}+C\norm{|w|}^{2}\norm{w_{tt}}^{2}+ C w_{t}^{2}(0,t)\norm{w_{tx}}^{2}.
	\end{align*}
	The last term on the right hand side of \eqref{2.25} gives
	\begin{align*}
		-w_{d}(w_{ttx},w_{tt})\leq C\norm{w_{tt}}^{2}+ \frac{\nu}{4}\norm{w_{ttx}}^{2}.
	\end{align*}
	Therefore, from \eqref{2.25}, it follows that
	\begin{align*}
		\frac{1}{2}\frac{d}{dt}\norm{w_{tt}}^{2}+\frac{\nu}{2} \norm{w_{ttx}}^{2}&+ (\frac{c_{0}}{2}+w_{d})w_{tt}^{2}(0,t)+(c_{1}+w_{d})w_{tt}^{2}(1,t)\nonumber\\&+\frac{2}{6c_{0}}w^{2}(0,t)w_{tt}^{2}(0,t)+\frac{2}{6c_{1}}w^{2}(1,t)w_{tt}^{2}(1,t)\nonumber\\&\leq \frac{4}{3c_{0}}w_{t}^{4}(0,t)+\frac{4}{3c_{1}}w_{t}^{4}(1,t)+C\norm{|w|}^{2}\norm{w_{tt}}^{2}+ C w_{t}^{2}(0,t)\norm{w_{tx}}^{2}.
	\end{align*}
	Multiplying by \(2e^{2\alpha t}\)  and integrating with respect to time over \([0,t]\) to the resulting inequality, we obtain using Theorem \ref{L2.12}
	\begin{align*}
		e^{2\alpha t}\norm{w_{tt}}^{2} + \int_{0}^{t} e^{2\alpha s} \Big( \nu \norm{w_{ttx}}^{2} + E_{4}(s)\Big)ds\leq C( \norm{w_{0}}_{3}) \exp(C\norm{w_{0}}_{2}^{2})+\norm{w_{tt}(0)}^{2}.
	\end{align*}
	Differentiation \eqref{w11} with respect to time, we get after putting \(t=0\)
	\begin{align*}
		\norm{w_{tt}(0)}\leq C\left( \norm{w_{0}}_{3}+\norm{w_{0}}_{1}\norm{w_{0}}_{2}\right).
	\end{align*}
	We complete the proof by multiplying \( e^{-2\alpha t}\).
\end{proof}
\textbf{Part 3.} In this part, we show $w_{ttx}\in L^{\infty}(0,T; L^{2}(0,1))$ and $w_{ttt}\in L^{2}(0,T;L^{2}(0,1))$
\begin{align*}
	\nu \norm{w_{ttx}}^{2}+ E_{5}(t)+ e^{-2\alpha t}\int_{0}^{t}e^{2\alpha s}\norm{w_{ttt}(s)}^{2}ds\leq e^{-2\alpha t}C(\norm{w_{0}}_{3}) \exp(C\norm{w_{0}}_{2}^{2}),
\end{align*}
where \[ E_{5}(t)= (c_{0}+w_{d})w_{tt}^{2}(0,t)+(c_{1}+w_{d})w_{tt}^{2}(1,t)+\frac{1 }{3c_{0}}w^{2}(0,t)w_{tt}^{2}(0,t)+\frac{1}{3c_{1}}w^{2}(1,t)w_{tt}^{2}(1,t).
\]
\begin{proof}
   The proof follows from parts \(1\) and \(2.\)
\end{proof}

\end{document}